\documentclass[a4paper,11pt]{article}

\usepackage{amsmath}
\usepackage{amsfonts}
\usepackage{amssymb}
\usepackage{graphicx}
\usepackage{dsfont}
\usepackage{amsthm}
\usepackage{mathrsfs}
\usepackage{ bbm }
\usepackage[english]{babel}
\usepackage{esint}
\usepackage{fancyhdr}
\usepackage{color}
\usepackage{comment}
\usepackage{tikz}

\usetikzlibrary{calc}
\usetikzlibrary{graphs}
\usetikzlibrary{intersections}
\usetikzlibrary{patterns}
\usetikzlibrary{plotmarks}
\usetikzlibrary{shapes}
\usepackage{fancyhdr}
\theoremstyle{plain}
\newtheorem{teorema}{Theorem}
\newtheorem{lemma}{Lemma}
\newtheorem{corollario}{Corollary}
\theoremstyle{definition}
\newtheorem{definizione}{Definition}
\newtheorem{example}{Example}
\newtheorem{remark}{Remark}
\newcommand{\numberset}{\mathbb}

\newcommand{\R}{\numberset{R}}
\newcommand{\bmo}{B\!M\!O}

\graphicspath{{figuri/}}

\def\Xint#1{\mathchoice
{\XXint\displaystyle\textstyle{#1}}%
{\XXint\textstyle\scriptstyle{#1}}%
{\XXint\scriptstyle\scriptscriptstyle{#1}}%
{\XXint\scriptscriptstyle\scriptscriptstyle{#1}}%
\!\int}
\def\XXint#1#2#3{{\setbox0=\hbox{$#1{#2#3}{\int}$ }
\vcenter{\hbox{$#2#3$ }}\kern-.6\wd0}}

\def\dashint{\Xint-}

\begin{document}

\title{Extension  and embedding theorems for Campanato spaces on $C^{0,\gamma}$ domains}

\author{Damiano Greco\footnote{Swansea University,  Fabian Way, Swansea,Uk, e-mail: damiano.greco@swansea.ac.uk}\quad and\quad    Pier Domenico Lamberti\footnote{Dipartimento Tecnica e Gestione dei Sistemi Industriali, University of Padova, Stradella S. Nicola 3,  36100 Vicenza, Italy, e-mail: lamberti@math.unipd.it}}

\date{\today}

\maketitle

\begin{abstract}  We consider  Campanato spaces  with exponents $\lambda , p$ on domains of class $C^{0,\gamma}$  in the N-dimensional Euclidean space endowed with a natural anisotropic metric depending on $\gamma$. We discuss several results including the appropriate Campanato's embedding theorem and we prove that functions of those spaces can be extended to the whole of the Euclidean space without deterioration of the exponents $\lambda, p$.
\end{abstract}


\section{Introduction}

Given an open set  $\Omega$ in $\R^N$,  $1\le p<\infty $ and $\lambda >0$ the  Campanato space ${{\mathcal L}^{\lambda  }_{p}(\Omega ) }$  is the space of functions $f\in L^p_{loc}(\Omega)$ such that  the seminorn defined by 
\begin{equation}\label{clasintro}
\sup_{B  }\left(\frac{1}{|B\cap \Omega|^{\lambda}  }\int_{B\cap \Omega}|f(y)    
-\dashint_{B\cap \Omega}   f(z)dz  |^pdy \right)^{\frac{1}{p}}
\end{equation}
is finite. Here the supremum is taken over all balls $B$ with center in $\Omega$ and radius not exceeding the diameter of $\Omega$ and $|B\cap \Omega|$ denotes the Lebesgue measure of $B\cap\Omega$.  The importance of Campanato spaces is well-known in the literature, in particular because  they include other classes of function spaces, such as Morrey spaces if $0<\lambda <1$, $\bmo$ spaces if $\lambda =1$ and spaces of H\"{o}lder continuous functions if $\lambda >1$. This equivalence is established by the celebrated Campanato's embedding Theorem under suitable assumptions on the open set $\Omega$, see Theorem~\ref{daprato} for a statement.  

It is important to note that the   space $\bmo (\Omega)$  of functions with bounded mean oscillation can be formally defined in the same way by setting $p=\lambda =1$ in \eqref{clasintro}. However, in the classical definition
one takes the supremum over all balls {\it contained} in $\Omega$.   This subtle distinction is obviously not relevant if $\Omega = \R^N$ but also if $\Omega$ is a bounded open set with Lipschitz boundary or the half space: in these cases one can identify the two spaces by setting  $\bmo (\Omega)= {{\mathcal L}^{1  }_{1}(\Omega ) } $, see Lemma~\ref{equivalence}. However, in general the two spaces are different, see Example \ref{strip}.

It is impossible to give an account of all possible applications of Campanato spaces but  we would like to highlight the fact that  the integral characterisation of  
H\"{o}lder continuous functions by means of the seminorm \eqref{clasintro} is of fundamental importance in the study of apriori estimates  for elliptic and parabolic equations, in particular in connection with the variational approach developed  by Ennio De Giorgi\footnote{In fact, it is acknowledged in \cite{campa63, campa64} that the study of  these function spaces was suggested to Sergio Campanato by Ennio De Giorgi. This historical fact doesn't seem much known outside Italy, probably because papers \cite{campa63, campa64} are written in Italian.} for the solution of the nineteenth Hilbert's problem.  We refer to \cite{beveza}  for a recent article  on apriori estimates in Campanato spaces including further references. 

In this paper, we discuss the  problem of extending a function $f\in {\mathcal L^{\lambda  }_{p}(\Omega ) }$ to the whole of $\R^N$ by  a function $Tf$ belonging to  $ {\mathcal L}^{\lambda  }_{p}(\R^N ) $. The focus is  on open sets $\Omega$ with boundaries of class $C^{0,\gamma}$, $\gamma \in ]0,1]$.  This means that $\Omega$ is locally described  at the boundary  as the subgraph of a H\"{o}lder continuous function with exponent $\gamma$.  In particular, if   $\gamma <1$  the boundary may exhibits singular points such as  power-type cusps, see Figure~\ref{holder_patchy}. 
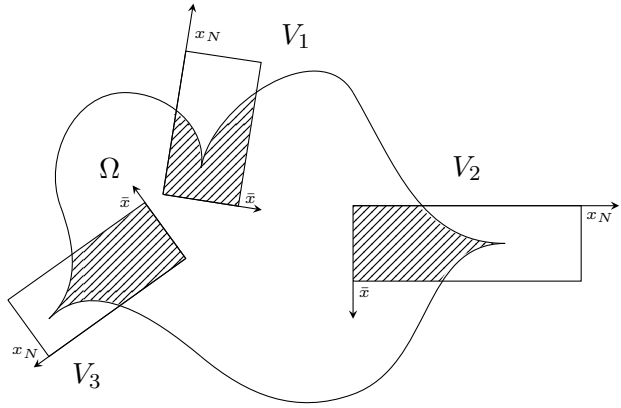
\begin{figure}\label{holder_patchy}
\centering
\begin{tikzpicture}

\draw (2,-2) to [out=80,  in=120] (4,-1);
\draw (4,-1) to [out=-60, in=180] (6,-3);
\draw (6,-3) to [out=-180, in=20 ] (4,-5);
\draw (4,-5) to [out=-160,in=-40] (2,-4.5);
\draw (2,-4.5) to [out=-220, in= 45] (0,-4);
\draw (0,-4) to [out=45,  in=-70]  (0.15, -2.5);
\draw  (0.15, -2.5) to [out=110, in=180]      (1,-1);
\draw (1,-1) to [out=0 , in= 80] (2,-2);

\pattern [pattern=north east lines]  (4,-3.5) 
to [out=0, in=180](5.15,-3.5)
to [out=60, in=180  ](6,-3)
to [out=180, in=-45 ](4.86,-2.5)
to [out=180, in=0](4,-2.5);

\pattern [pattern=north east lines] (1.5,-2.35)
to [out=80,  in=-93] (1.65,-1.28)
to [out=-20 , in=90 ] (2,-2)
to [out=80, in=-125] (2.73,-0.95)
to [out=-100, in=80] (2.491,-2.507);

 \pattern [pattern=north east lines] ( 1.8,-3.2)   
to [out=120,  in=-60] (1.26,-2.46) 
to [out=220 , in=40 ]  (0.32,-3.15)
to [out=-90, in=45] (0,-4)
to [out=42, in=158] (0.95,-3.82);

\draw (4,-3.5) -- (4,-2.5) --  (7,-2.5) -- (7,-3.5)-- (4,-3.5); 
\draw  (1.5,-2.35) -- (1.8,-0.45) --  (2.791,-0.6) -- (2.491,-2.507)-- (1.5,-2.35); 
\draw (-0.54,-3.75)  -- (1.26,-2.46) -- ( 1.8,-3.2)-- (0.0,-4.5)--(-0.54,-3.75) ;
\node at (3.25,-0.25) {$V_1$};
\node at (0.8, -2) {$\Omega$};
\node at (0.5,-4.75) {$V_3$};
\node at (5.5, -2) {$V_2$};

\draw [-stealth]  (4,-3.5) -- (4,-4) ;
\draw [-stealth]  (4,-2.5) -- (7.5,-2.5) ;
\node at  (4.15,-3.65) {\tiny{$\bar{x}$}};
\node at (7.25,-2.65) {\tiny{$x_N$}};

\draw[-stealth] (1.5,-2.35)--(2.8,-2.55526);
\node at (2.65,-2.4) {\tiny{$\bar{x}$}};
\draw [-stealth] (1.5,-2.35)--(1.9,12.03-711/60) ;
\node at (2.1,11.6-711/60) {\tiny{$x_N$}};

\draw[-stealth]( 1.8,-3.2) --(1.1,-2.23);
\draw [-stealth]  ( 1.8,-3.2) --(-0.2,-4.644);
\node at (1,-2.45) {\tiny{$\bar{x}$}};
\node at (-0.3,-4.45) {\tiny{$x_N$}};
\end{tikzpicture}
\caption{A domain of class $C^{0,\gamma}$.}
\end{figure}
The case $\gamma =1$ corresponds  to  the class of domains with Lipschitz continuous boundaries which, from the point of view of extension theory, represents the regular case and is already discussed in the literature, see  \cite{jones, stri} for the case $\lambda=p=1$.   We refer to \cite{gallia} for recent results on the extension for $\bmo$ spaces and their variants on uniform domains, as well as to \cite{bib2} for a remarkable decomposition theorem for the space $\bmo $ on arbitrary domains.

We note  from the very beginning that  the case $\lambda <1$ is the easy one, at least in principle. Indeed, for $\lambda <1$ Campanato spaces coincide with Morrey spaces and it is simple to prove that functions in the Morrey spaces can be extended by zero to the whole of $\R^N$, see Lemma~\ref{extensionbyzero}. Also the case $\lambda >1$ is not difficult if we identify Campanato spaces with spaces of H\"{o}lder continuous functions: these functions can be extended by means of the Bj\"{o}rk's Theorem (see e.g. \cite[Thm.~1.8.3]{bib4}). On the other hand, it is well-known that the extension problem for $\bmo$ functions is highly non-trivial. Note that in general  $\bmo$ functions cannot be extended by zero. The classical example is provided by the function $f(x)=\log x$, with $x>0$, considered as a function defined on $(0,\infty)$. The  extension by zero of $f$ defined  by setting $f(x)=0$ for  all $x<0$ does not belong to $\bmo (\R)$, see Example~\ref{log_0}.   

The extension problem for classical $\bmo$ spaces (where balls $B$ in \eqref{clasintro} are contained in $\Omega)$  was  studied in the fundamental paper  \cite{jones} by Peter Jones who  characterised the domains which allow the extension of $\bmo$ functions. These are known in the literature as uniform domains and include bounded domains with Lipschitz boundaries but not domains of class $C^{0,\gamma}$  if $\gamma <1$.  Up to our knowledge, not much is available in the literature in this case. 

As discussed in  \cite{lamves}, the pioneering analysis carried out in \cite{barozzi} and \cite{dapra} suggests that in the case of domains of class $C^{0,\gamma}$ the Euclidean metric  should be replaced by a metric depending on $\gamma$.  This metric, denoted by  $\delta_{\gamma}$,  is anisotropic and its definition depends on the direction of the possible cusps of the boundary. For example, if the boundary of $\Omega$ is represented by the cusp of equation $x_N=|\bar x |^{\gamma}$ where the elements   of ${\mathbb R}^N$ for $N\geq 2$ are written in the form $x=(\overline x,x_N)$, with $\overline x=(x_1,\dots , x_{N-1})\in{\mathbb R}^{N-1}$ and $x_N\in {\mathbb R}$, then the natural metric to be used is  
\begin{equation}\label{metricgamma}
\delta_{\gamma}(x,y)= \max\{ |\bar x-\bar y|^{\gamma}, |x_N-y_N| \}\, ,
\end{equation}
for all $x,y\in {\mathbb R}^N$.  It is important to note that the Lebesgue measure of the corresponding balls $B_{\gamma}$ of radius $r$ behaves like 
$r^{N_{\gamma}}$ where 
\begin{equation}\label{critical_exp}
N_{\gamma}= \frac{N-1}{\gamma}+1  .
\end{equation}
Moreover, if $\Omega$ is of class $C^{0,\gamma}$ also the Lebesgue measure of $ B_{\gamma}\cap \Omega$ behaves like $r^{N_{\gamma}}$: this property plays a crucial role in the analysis of Campanato spaces on domains and is called Property (A) in the literature.

By replacing the balls $B$ in \eqref{clasintro} by the anisotropic balls $B_{\gamma}$ we obtain the corresponding Campanato spaces 
${{\mathcal L}^{\lambda  }_{p,\gamma}(\Omega ) }$ depending on $\gamma$.
These spaces are better suited for the validity of embedding and extension theorems in the spirit of classical results. Indeed, by using the  general result of \cite{dapra} one can prove Campanato's embedding Theorem in the non-critical case $\lambda \ne 1$, see \cite{lamves}. On the other hand, in the critical case $\lambda =1$ it is possible to use the generalised John-Niremberg inequality from \cite{aalto} to prove that   the spaces ${{\mathcal L}^{1  }_{p,\gamma}(\Omega ) } $ are independent of $p$, a well-known fact in the Lipschitz case $\gamma=1$. See Theorem~\ref{daprato} for a detailed proof. 

The importance of using the metric $\delta_{\gamma}$ is also evident in the problem of extension. The simplest extension operator is provided by reflection. For example, given a cusp of the form $\Omega= \{ (\bar x, x_N)\in \R^N:\ x_N<|\bar x|^{\gamma}\}$ one would extend a function $f(\bar x, x_N)$ defined in $\Omega$ by setting 
\begin{equation}\label{exext}
Tf(\bar x , x_N)= f(\bar x,  2|\bar x|^{\gamma}-x_N   ),\ \  {\rm if} \ x_N>|\bar x|^{\gamma} .
\end{equation}
It is clear that the natural introduction of  the term $|\bar x|^{\gamma}$ in formula \eqref{exext} deteriorates the degree of smoothness of  $f$: for example, a Lipschitz function $f$  would be transformed into a H\"{o}lder continuous function of exponent $\gamma$. (This deterioration phenomenon is well-known and  is studied in \cite{bu} for Sobolev spaces of arbitrary order.) This example gives a strong hint about the fact that  incorporating $|\bar x|^{\gamma}$ in the metric as in \eqref{metricgamma} allows to keep the same smoothness exponent in the extension of functions.  

In this paper, we consider this problem for  Campanato spaces of arbitrary order $\lambda$ (including the theoretically  difficult case $\lambda =1$), not only for the sake of uniform treatment but also to emphasise  the preservation of the exponent $\gamma$ in the extension by reflection. 

First, we  consider the case of  elementary unbounded   domains of class $C^{0,\gamma}$. These  domains are defined as  subgraphs of H\"{o}lder continuous functions with exponent $\gamma$, defined on $\R^{N-1}$. As done in \cite{stri} for $\bmo $ functions in the Lipschitz case $\gamma =1$, in Theorem~\ref{estensioneholder} we prove that the even extension (that is, the reflection with respect to the boundary) preserves the Campanato spaces ${{\mathcal L}^{\lambda  }_{p,\gamma}}$ (note that the odd extension wouldn't work). 

Next, we consider the general case where different parts of the boundary have to be rotated in different directions in order to be represented as graphs of H\"{o}lder continuous functions with exponent $\gamma$, see Figure~\ref{holder_patchy}.
In this case the underlying metric has to be adapted to the different directions and this leads to the analysis of Campanato spaces  $ \mathcal{L}^{\lambda ,\mathcal R}_{p, \gamma }(\Omega) $ obtained as finite sums of Campanato spaces. Here ${\mathcal R}$ is the finite collection of rotations used to represent the boundary of $\Omega$ as mentioned above.

Note that for $\gamma =1$  the space $ \mathcal{L}^{\lambda ,\mathcal R}_{p, \gamma }(\Omega) $ is independent of ${\mathcal{R}}$ and coincides with the usual Campanato space since all metrics involved are strongly equivalent to the Euclidean metric, see Remark~\ref{patcheq}. On the other hand, since for $\gamma \ne 1$ the metrics arising from different directions are not strongly equivalent, we believe that it would be difficult (and probably not even particularly useful in applications) to give an intrinsic definition of the space  $ \mathcal{L}^{\lambda ,\mathcal R}_{p, \gamma }(\Omega) $ independent of ${\mathcal{R}}$.

We conclude the paper by formulating an extension problem  for a natural $\bmo_{\gamma}$  space associated with the metric $\delta_{\gamma}$ as in \eqref{bmogamma}, which we believe is open. 
 

We refer  to the classical monograph  \cite{bib4} for a clear introduction to the theory of Campanato spaces and to  \cite{samko} for a recent survey. We also quote \cite{bib12} as a standard reference  for  $\bmo$ spaces and \cite{sawano} for an extensive study of Morrey spaces. Finally, we quote  papers \cite{fanlam}, \cite{lamves}, \cite{lamvio} for related recent  results concerning the extension problem for  Sobolev-Morrey spaces. 

This paper is organized as follows. Section 2 is  devoted to preliminary and auxiliary results which have also their own interest. In particular, we discuss Campanato's embedding Theorem, the relation between $\bmo$ spaces and Campanato spaces and the multiplication by $C^{\infty}$-functions in Campanato spaces. In Section 3 we discuss the extension problem,  we prove Theorem~\ref{estensioneholder} concerning the extension of Campanato spaces on elementary domains; moreover,  we define the Campanato spaces  $ \mathcal{L}^{\lambda ,\mathcal R}_{p, \gamma }(\Omega) $ and we prove the corresponding extension theorem; we also formulate an open problem.

\section{Preliminary results and embeddings}

\subsection{Notation}

As mentioned in the introduction, we denote  the elements of ${\mathbb R}^N$ for $N\geq 2$ by  $x=(\overline x,x_N)$, with $\overline x=(x_1,\dots , x_{N-1})\in{\mathbb R}^{N-1}$ and $x_N\in {\mathbb R}$. 
For  any   $\gamma\in ]0,1]$,  we   consider the metric  $\delta_{\gamma}$ in ${\mathbb R}^N$ defined by \eqref{metricgamma}.
The corresponding open balls  with  radius $r$, centred at $x$  are defined  by
\begin{eqnarray*}
B_{\gamma}(x,r)=\{y\in {\mathbb R}^N:\ \delta_{\gamma}(x,y)<r   \} .
\end{eqnarray*}
It is useful to observe that the Lebesgue measure  of $B_{\gamma}(x,y)$ equals   
$
2\omega_{N-1}r^{N_{\gamma}}    
$
where $\omega_{N-1}$ is the measure of the unit ball in ${\mathbb R}^{N-1}$ and $N_{\gamma}$ is defined in 
\eqref{critical_exp}.
The diameter of a set $A$ in ${\mathbb{R}}^N$ with respect to the metric $\delta_{\gamma}$ is denoted by $\delta_{\gamma}(A)$. Namely, 
$\delta_{\gamma}(A)=\sup\{\delta_{\gamma}(x,y):\ x,y\in A  \}$. Finally, by $\delta_{\gamma}(A,B)$ we denote the distance between two sets $A,B$ in $\R^N$ defined by  $\sup\{\delta_{\gamma}(x,y):\ x\in A, \  y\in B\}$.

Let $\Omega$ be an open set in ${\mathbb R}^N$.  To avoid taking care of minor details, we shall always  assume that $\Omega$ is a domain, which means that $\Omega$ is a connected open set.     If $f$ is a real-valued function defined in $\Omega$, we denote by 
$f_0$ the extension by zero of $f$. Namely, 
$$
f_0(x)=\left\{
\begin{array}{ll}f(x),& \ {\rm if}\  x\in \Omega,\\
0,& \ {\rm if}\  x\in \mathbb{R}^N\setminus\Omega .
\end{array}
 \right. 
$$
Let $p\in [1,\infty [$ and  $\lambda \in ]0,\infty [$. If $f$ is such that $f_0\in L^p_{loc}({\mathbb{R}}^N)$ we set 
\begin{equation}\label{normmorrey}
\| f \|_{L^{\lambda }_{p, \gamma }(\Omega ) }:=\left(\sup_{x\in \Omega}\,\sup_{r\in \left]0,  \delta_{\gamma}(\Omega )\right]  }\frac{1}{|B_{\gamma}(x, r)\cap \Omega|^{\lambda}}  \int_{B_{\gamma}(x, r)\cap \Omega}|f(y) |^pdy \right)^{\frac{1}{p}}
\end{equation}
and 
\begin{equation}\label{semi1}
| f |_{{\mathcal L}^{\lambda  }_{p, \gamma }(\Omega ) }:=\sup_{x\in \Omega}\,\sup_{r\in \left]0,  \delta_{\gamma}(\Omega )\right]  } \!\!  \left(\frac{1}{|B_{\gamma}(x, r)\cap \Omega|^{\lambda}  }\int_{B_{\gamma}(x, r)\cap \Omega}  \biggl|f(y)    
-\dashint_{B_{\gamma}(x, r)\cap \Omega}   f(z)dz  \biggr|^pdy \right)^{\frac{1}{p}}\, .
\end{equation}
The corresponding Morrey spaces are defined by 
$$
L^{\lambda }_{p, \gamma }(\Omega )=\{ f:\Omega\to {\mathbb{R}}:\ f_0\in L^p_{loc}({\mathbb{R}}^N),\  {\rm and}\ \| f \|_{L^{\lambda }_{p, \gamma }(\Omega ) }< \infty   \} .
$$
Similarly,  the  Campanato spaces are defined by 
$$
{\mathcal L}^{\lambda}_{p, \gamma }(\Omega )=\{ f:\Omega\to {\mathbb{R}}:\ f_0\in L^p_{loc}({\mathbb{R}}^N),\ {\rm and}\  | f |_{{\mathcal L}^{\lambda }_{p, \gamma }(\Omega ) }< \infty   \} . 
$$
Note that if $\Omega$ is bounded then $L^{\lambda }_{p, \gamma }(\Omega ), {\mathcal L}^{\lambda}_{p, \gamma }(\Omega )\subset L^p(\Omega)$,  $\| \cdot \|_{ L^{\lambda }_{p, \gamma }(\Omega ) } $  defines a norm in  $L^{\lambda }_{p, \gamma }(\Omega )$  while $| \cdot |_{{\mathcal L}^{\lambda }_{p, \gamma }(\Omega ) } $ is a seminorm in ${\mathcal L}^{\lambda }_{p, \gamma }(\Omega )$. In order to have a normed space,  it is customary to endow  the Campanato space  ${\mathcal L}^{\lambda }_{p, \gamma }(\Omega )$ with the norm defined by 
$$
\| f \|_{{\mathcal L}^{\lambda  }_{p, \gamma }(\Omega ) } := \|f \|_{L^p(\Omega)}+ | f |_{{\mathcal L}^{\lambda  }_{p, \gamma }(\Omega ) } \, ,
$$
for all $f\in {\mathcal L}^{\lambda  }_{p, \gamma }(\Omega )$.

We observe  that  ${L^{\lambda }_{p, 1 }(\Omega ) }$,  ${{\mathcal{L}}^{\lambda }_{p, 1 }(\Omega ) }$ are 
the classical Morrey and Campanato spaces  and we  recall that ${L^{\lambda }_{p, 1 }(\Omega ) }$ contains only the zero function for $\lambda >1$ and it coincides with $L^{\infty}(\Omega)$ for $\lambda =1$.

\subsection{Campanato vs $\bmo$ spaces}
\label{bmosec}

In the literature the Campanato space $\mathcal{L}^{1}_{1,1}(\mathbb{R}^N)$ is better known as the space of functions with bounded mean oscillation  and is denoted by  $\bmo (\mathbb{R}^N)$.  In the case of a domain  $\Omega $ strictly contained in $\mathbb{R}^N$, the definition of $\bmo (\Omega )$ is not univocal. One classical definition is the following
\begin{equation} 
\bmo (\Omega ) =\left\{ f\in L^1_{loc}(\Omega):\     |f|_{\bmo (\Omega )} \ne \infty   \right\}
\end{equation}
where 
$$ \sup_{Q\subset \Omega}   \dashint_{Q} \left|f(y)    
-\dashint_{Q}   f(z)dz  \right|dy
$$ 
and the supremum is taken on all cubes $Q$ with sides parallel to the coordinate planes (or, equivalently, all Euclidean balls contained in $\Omega$), see e.g., \cite{jones}.   Other authors identify  $\bmo (\Omega )$ with  $\mathcal{L}^{1}_{1,1}(\Omega)$, see e.g., \cite{samko}. See also e.g., \cite{aalto} for a general definition in metric spaces.   From the definitions it's clear that $\mathcal{L}^{1}_{1,1}(\Omega)$ is a subspace of $\bmo(\Omega)$ no matter what $\Omega$ is.  Nevertheless,  there are cases where the two spaces coincide and cases where they don't.  

\begin{lemma}\label{equivalence}
Let $\Omega$ be a bounded Lipschitz domain in $\R^N$. Then, 
$$\bmo(\Omega)= \mathcal{L}^{1}_{1,1}(\Omega).$$
\end{lemma}
\begin{proof}
It's  enough to prove the inclusion $ \bmo(\Omega)\subset \mathcal{L}^{1}_{1,1}(\Omega)$. Under our assumptions on $\Omega$, by \cite[Theorem $1$]{jones} it follows that   every function $f\in \bmo(\Omega)$ can be extended to a function $\tilde{f}\in \bmo(\R^N)$.  On the other hand,   since $\Omega$ satisfies condition \eqref{lemmacusp1} with $\gamma=1$,  it's easy to see that the restriction $\tilde{f}|_{\Omega}=f$ belongs to $ \mathcal{L}^{1}_{1,1}(\Omega)$ (see Lemma \ref{restrizioneholder} below). 
\end{proof}
\example \label{strip} Let $\Omega=\R\times ]-1,1[$ and $f(x_1,x_2)=x_1$.  It's easy to see that $f\in \bmo(\Omega)$ (see \cite[Remark $1$]{bolk}).  However,  if we consider a family of squares  $Q=]-r,r[^{\,2}$ centred at the origin and having edges of lenght $2r$,  $r>1$  we obtain
\begin{equation*}
\frac{1}{|Q\cap \Omega|}\int_{Q\cap \Omega}|f-f_{Q\cap \Omega}|=\frac{1}{2r}\int^{r}_{-r}|x_1|dx_1=\frac{r}{2}
\end{equation*}
which diverges as $r$ goes to infinity. Thus $f\notin  \mathcal{L}^{1}_{1,1}(\Omega)$.

\subsection{Embedding theorems and equivalent norms}

Most results of this paper are based on the validity of the so-called condition (A) on the domain $\Omega$. For $\gamma =1$ this terminology was used already in the paper \cite{campa63}. 
 Namely, it is required that 
 there exists  a positive constant $c$ independent of $x$ such that 
\begin{equation}
 \label{lemmacusp1}
|B_{\gamma}(x,r)\cap \Omega| \geq cr^{N_{\gamma}},
\end{equation}
for all $x\in \overline{\Omega}$ and $r\in ]0,\delta_\gamma(\Omega)]$.   

\begin{remark}If the domain  $\Omega$ satisfies inequality \eqref{lemmacusp1},  we can define the spaces $\mathcal{L}^{\lambda}_{p,\gamma}(\Omega)$ (respectively $L^{\lambda}_{p,\gamma}(\Omega)$) by replacing the quantity  $|B_\gamma(x,r)\cap \Omega|^{\lambda}$ by $r^{\lambda N_\gamma }$ in \eqref{semi1} (respectively \eqref{normmorrey}).
\end{remark}

 In particular, under assumption \eqref{lemmacusp1}   Campanato's embedding Theorem holds. Here 
$C^{0,\alpha}(\overline{\Omega},\delta_\gamma)$ denotes the space of H\"{o}lder continuous functions with respect to the metric $\delta_{\gamma}$,  with exponent $\alpha\in ]0,1]$, endowed with the norm
$$
\| f\|_{ C^{0,\alpha}(\overline{\Omega},\delta_\gamma) }=\sup_{x\in\overline{\Omega} }|f(x)|+\sup_{x,y\in \overline{\Omega}}\frac{|f(x)-f(y)|}{\delta_{\gamma}^{\alpha}(x,y)}\, .
$$

\begin{teorema}(Campanato-Da Prato-John-Nirenberg)\label{daprato}
Let $\Omega$ be a bounded domain in  $\R^N$ satisfying condition \eqref{lemmacusp1}. The following statements hold:
\begin{itemize}
\item[$(i)$] If $0<\lambda<1$ then $\mathcal{L}^{\lambda}_{p,\gamma}(\Omega)\simeq L^{\lambda}_{p,\gamma}(\Omega)$;
\item[$(ii)$] If $\lambda =1$ then $\mathcal{L}^{\lambda}_{p,\gamma}(\Omega)\simeq  \mathcal{L}^{\lambda}_{1,\gamma}(\Omega) $ for all $p\in [1, \infty [$;
\item[$(ii)$] If $\lambda>1$ then $\mathcal{L}^{\lambda}_{p,\gamma}(\Omega)\simeq C^{0,\alpha}(\overline{\Omega},\delta_\gamma)$
where 
$\alpha=\frac{N_\gamma(\lambda-1)}{p}.$
\end{itemize}
\end{teorema} 

 It is understood that if $\gamma \alpha >1$ the space $C^{0,\alpha}(\overline{\Omega},\delta_\gamma)$
consists of constant functions only. 
Here the symbol  $ \simeq $ is used to indicate that two spaces coincide  (possibly identifying functions which coincide  almost everywhere) and the corresponding norms are equivalent.  
The proofs of statements $(i)$ and $(iii)$ in the previous  theorem can be found in \cite{dapra}. The proof of statement $(ii)$ can be performed  by using the celebrated John-Niremberg inequality and (in a form or another) is known  in the literature for the Euclidean metric, see e.g., \cite[Thm. 4.3]{samko} and \cite[Thm.~14]{bolk}.     Since we do  not have a specific reference for a proof of statement $(ii)$ above for $\gamma\ne 1$, we find it convenient to include it here. \\

 \begin{proof}   We prove  statement $(ii)$ in Theorem~\ref{daprato}. 
Let $f\in\mathcal{L}^{1}_{p,\gamma}(\Omega)$.  By H\"older's inequality, we deduce  that 
\begin{equation}\label{holder}
 \dashint_{B_\gamma(x,r)\cap \Omega}|f(y)-f_{B_\gamma(x,r)\cap \Omega}|  dy
 \le \left(\dashint_{B_\gamma(x,r)\cap \Omega}|f(y)-f_{B_\gamma(x,r)\cap \Omega}|^p dy\right)^{\frac{1}{p}},
\end{equation}
which proves that $\mathcal{L}^{1}_{p,\gamma}(\Omega)\subset  \mathcal{L}^{1}_{1,\gamma}(\Omega)$. The corresponding estimate for the norms which proves that this embedding is continuous can be easily deduced by the previous inequality and H\"{o}lder's inequality. 

We now prove the  reverse  inclusion. Let  $f\in   \mathcal{L}^{1}_{1,\gamma}(\Omega)$.    Then,  by the general result in \cite[Theorem $5.2$]{aalto} applied to the metric space $(\Omega ,{ \delta_{\gamma}}_{|\Omega }   )$, it follows that 
 $$\left|\left\{x\in B_\gamma:\left|  f(x)- f_{B_\gamma}\right|>t\right\}\right|\le c_1 e^{-\frac{c_2 t}{a }}|B_\gamma|$$
where $a = |f|_{ \mathcal{L}^{1}_{1,\gamma}(\Omega)}$, $B_{\gamma}$ is any ball in the metric space $(\Omega , {\delta_{\gamma}}_{|\Omega})$ and $c_1,c_2$  are positive constants independent of $ f$, $t$ and $B_{\gamma}$.   (Note that since $\Omega$ satisfies \eqref{lemmacusp1} then the measure  ${\delta_{\gamma}}_{|\Omega}$  is doubling as required in \cite{aalto}.) Thus, 
\begin{equation}\label{john}
\begin{split}
& \frac{1}{|B_\gamma  |}\int_{B_\gamma }| {f}(y)-  {f}_{B_\gamma }|^pdy\\
&=\frac{p}{|B_\gamma |}\int_{0}^{\infty}t^{p-1}|\left\{y\in B_\gamma: |  {f}(y)-  {f}_{B_\gamma }|>t\right\}| dt\\
& \le c_1 p \int_{0}^{\infty} t^{p-1}e^{-\frac{c_2 t}{a }}  dt=\frac{c_1}{c^p_2}\Gamma(p+1) |f|_{ \mathcal{L}^{1}_{1,\gamma}(\Omega)  }^p   =c_3
 |f|_{ \mathcal{L}^{1}_{1,\gamma}(\Omega)  }^p  ,
\end{split}
\end{equation}
where we have set $c_3=c_1c^{-p}_2\Gamma(p+1)$.  In order to estimate the $L^p$-norm of the function $f$, we simply use triangle inequality to deduce by 
\eqref{john} that 
\begin{eqnarray}\label{john1} \| f\|_{L^p(B_{\gamma})}\le  |B_{\gamma}|^{1/p}  \left(\frac{\| f \|_{L^1(B_{\gamma})}}{|B_{\gamma}|}   +c_3^{1/p}  |f|_{ \mathcal{L}^{1}_{1,\gamma}(\Omega)  }   \right) .
\end{eqnarray}
By \eqref{john} and inequality  \eqref{john1} applied to a ball with radius sufficiently large to ensure\footnote{recall that the balls $B_{\gamma}$ used in this part of the proof are restrictions to $\Omega$ of balls in ${\mathbb{R}}^N$ endowed with the metric $\delta_{\gamma}$} that $B_{\gamma}=\Omega$,  we deduce that $ \mathcal{L}^{1}_{1,\gamma}(\Omega) \subset \mathcal{L}^{1}_{p,\gamma}(\Omega)$, the embedding being continuous. 
\end{proof}

The previous theorems points out  also the fact that the spaces $\mathcal{L}^{\lambda}_{p,\gamma}(\Omega)$ heavily depends on $\gamma$. The following example clarifies this  by showing that in general 
 $\mathcal{L}^{\lambda}_{p,\gamma}(\Omega)\neq \mathcal{L}^{\lambda}_{p,1}(\Omega)$ if $\gamma\neq 1$.
\example Let $\Omega$ be the subset of $\R^2$ defined by 
$$\Omega=\left\{(x_1,x_2)\in \R^2: 0<x_2<1-|x_1|^{\gamma} \right\} $$
where $\gamma\in ]0,1[$. Let $1<\lambda\le 1+\frac{1}{\gamma+1}$.  If we choose $\alpha=(1+\tfrac{1}{\gamma})(\lambda-1)$,  Theorem~\ref{daprato} ensures  that the function $f(x)=\text{sign}(x_1)|x_1|^{\gamma \alpha}$ belongs to $\mathcal{L}^{\lambda}_{1,\gamma}(\Omega)$.  Note that this is true since $\Omega$ satisfies property \eqref{lemmacusp1},  $\gamma \alpha\le 1$ by our assumptions on $\lambda$,  and $f\in C^{0,\alpha}(\overline{\Omega},\delta_\gamma)$.  We now claim that $f\notin \mathcal{L}^{\lambda}_{1,1}(\Omega)$. We can prove this by using two arguments. \\
The first argument exploits an explicit computation.  Let $Q=]-r,r[\times ]0,2r[$ be the  family of squares centred at the point $(0,r)$ having radius $r$ small enough in order to guarantee that  $Q\subset \Omega$.  Then,  $f_{Q}=0$ (since $f$ is odd with respect to the variable $x_1$) and 
\begin{equation}
\frac{1}{|Q|^{\lambda}}\int_{Q}|f(x)|\,dx=\frac{4r}{|Q|^{\lambda}}\int_{0}^{r}|x_1|^{(\gamma+1)(\lambda-1)}\, dx_1=C(\lambda,\gamma)r^{(\gamma-1)(\lambda-1)}
\end{equation}
which diverges as $r$ goes to zero. This proves that $|f|_{\mathcal{L}^{\lambda}_{1,1}(\Omega)}=\infty$. \\
The second argument  is based on the classical Campanato's embedding theorem (with the Euclidean metric). Indeed, if we assume that $f\in \mathcal{L}^{\lambda}_{1,1}(\Omega )$ then by Theorem~\ref{daprato} with $\gamma=1$ we would obtain that $f\in C^{0,\beta}_{loc} ( \Omega,\delta_1  )$ where 
$\beta = 2(\lambda -1)$. Since $\beta > \gamma \alpha$, we clearly see  a contradiction from the definition of $f$. \\

In what follows, we shall often use  the  seminorm  
    \begin{equation}\label{eqsem}\left|f\right|^{(1)}_{\mathcal{L}^{\lambda}_{p,\gamma}(\Omega)}:=\left(\!\sup_{x\in\Omega}\,\sup_{r\in \left]0,  \delta_{\gamma}(\Omega )\right]  } \frac{1}{|B_{\gamma}(x,r)\cap\Omega|^{\lambda-1}}   \dashint_{B_\gamma(x,r)\cap\Omega}\dashint_{B_\gamma(x,r)\cap\Omega} \!\!  |f(y)-f(z)|^p dzdy\!\right)^{\frac{1}{p}},
    \end{equation}
 which turns out to be equivalent to \eqref{semi1}.   Indeed, the following lemma holds. 
 \begin{lemma}
\label{normeeq} Let $\Omega$ be a  domain in $\R^N$.
Then
$$ |f|_{\mathcal{L}^{\lambda}_{p, \gamma }(\Omega)}\le \left|f\right|^{(1)}_{\mathcal{L}^{\lambda}_{p, \gamma } (\Omega)}\le 2 |f|_{\mathcal{L}^{\lambda}_{p, \gamma } (\Omega)} $$ 
for all real-valued functions $f$ defined on $\Omega$ such that $f_0\in L^p(\Omega)$. 
\end{lemma}
\begin{proof}  Let $f$ be as in the statement and  $B_{\gamma}$ be any ball centred in $\Omega$ with radius $r\le \delta_{\gamma}(\Omega)$.  By Jensen's inequality  we have
\begin{eqnarray}\lefteqn{
\frac{1}{|B_{\gamma}\cap\Omega|^{\lambda}}\int_{B_{\gamma}\cap\Omega}\left|f(y)-f_{B_{\gamma}\cap\Omega}\right|^p dy}\nonumber\\
& & =\frac{1}{|B_{\gamma}\cap\Omega|^{\lambda}}\int_{B_{\gamma}\cap\Omega}\left|\dashint_ {B_{\gamma}\cap\Omega}(f(y)-f(z))dz\right|^p dy\nonumber\\
& &  \le \frac{1}{|B_{\gamma}\cap\Omega|^{\lambda-1}}\dashint_{B_{\gamma}\cap\Omega}\dashint_{B_{\gamma}\cap\Omega}|f(y)-f(z)|^p dzdy\nonumber\\
\end{eqnarray}
which proves one of the inequalities in the statement.
For the other inequality, we easily see that 
\begin{eqnarray}\lefteqn{
\dashint_{B_{\gamma}\cap\Omega}\dashint_{B_{\gamma}\cap\Omega}|f(y)-f(z)|^p dzdy} \nonumber \\
& &\le  \dashint_{B_{\gamma}\cap\Omega}\dashint_{B_{\gamma}\cap\Omega}\left(|f(y)-f_{B_{\gamma}\cap\Omega}|+|f(z)-f_{B_{\gamma}\cap\Omega}|\right)^p dzdy\nonumber\\
& & \le 2^{p-1} \dashint_{B_{\gamma}\cap\Omega}\dashint_{B_{\gamma}\cap\Omega} \left( |f(y)-f_{B_{\gamma}\cap\Omega}|^p + |f(z)-f_{B_{\gamma}\cap\Omega}|^p\right)dzdy\nonumber \\
& & \le 2^p |B_{\gamma}\cap\Omega|^{\lambda-1} \dashint_{B_{\gamma}\cap\Omega}|f|^p_{\mathcal{L}^{\lambda}_{p,\gamma}(\Omega)}dz 
 = 2^p |B_{\gamma}\cap\Omega|^{\lambda-1}  |f|^p_{\mathcal{L}^{\lambda}_{p,\gamma}(\Omega)}.
\end{eqnarray}
\end{proof}

\subsection{Multiplication in Campanato spaces}

Campanato spaces  $\mathcal{L}^{\lambda}_{p,\gamma}(\Omega)$ are closed under multiplication by smooth functions. This is simple to prove for Morrey and H\"{o}lder spaces, but is not straightforward  for $\lambda =1$. A proof for $\bmo (\R^N)$ can be found in \cite[Lemma~11]{bolasi}. 
Here we give an alternative proof valid for all $\gamma \in ]0,1]$ and $p\in [1,\infty [$. 

In the next result  we use the following condition on the domain $\Omega$:
\begin{equation}\label{inf_ball}
\mathcal{
I}(r):=\inf_{x\in \Omega}|B_\gamma(x,r)\cap \Omega|>0\quad \forall\ r\in ]0,\delta_\gamma(\Omega)].
\end{equation}
Note that this condition includes for example domains satisfying property \eqref{lemmacusp1}  in which case   $\mathcal{I}(r)\ge cr^{N_\gamma}$.
Note also that the condition $ \gamma \alpha \le 1 $ is used here only when $\lambda >1$ in order to avoid the trivial case where  all functions in $\mathcal{L}^{\lambda}_{p,\gamma}(\Omega)$ are constant. 
\begin{lemma}\label{product} Let $p\in [1,\infty[  $, $\gamma\in ]0,1]$, $\lambda \in ]0,\infty [$.
  If $\lambda >1$ assume also  that $\gamma \alpha \le 1 $ where $\alpha $ is defined in Theorem~\ref{daprato}. Let $\Omega$ be a domain  in $\R^N$ satisfying \eqref{inf_ball}. Let 
  $f\in \mathcal{L}^{\lambda}_{p,\gamma}(\Omega)$ and $\varphi\in  C^{\infty}(\Omega)$.  
Suppose that either $f$ or $\varphi$ has support contained in a compact subset $K$ of $\Omega$.
Then there exists a positive constant $C$  such that   
\begin{equation}\label{product_final}
|f\varphi|_{ \mathcal{L}^{\lambda}_{p,\gamma}(\Omega)}\le C\left(| f|_{ \mathcal{L}^{\lambda}_{p,\gamma}(\Omega)}+\|f\|_{L^{p}(U)}\right)
\end{equation}
where $U$  is a suitable bounded neighborhood of $K $  which depends only on $K$ and the distance $\delta_{\gamma}(K, \partial \Omega)$.
In particular,  if $f\in L^{p}(\Omega)$ then 
\begin{equation}\label{product2}
\left\|f\varphi\right\|_{\mathcal{L}^{\lambda}_{p,\gamma}(\Omega)}\le C\left\| f\right\|_{ \mathcal{L}^{\lambda}_{p,\gamma}(\Omega)}.
\end{equation}
The constant $C$ depends  only on  $N,p, \lambda, \gamma, \Omega,  \mathcal{I}, K$ and $\varphi$. 
\end{lemma}
\begin{proof}  Let   $U$ be a bounded domain satisfying property \eqref{lemmacusp1}  and such that $K\subset U\subset \Omega$. Let $$\bar{r}:=\frac{\min\left\{\delta_\gamma(K, \partial U), \delta_\gamma(U, \partial \Omega)\right\}}{4}.$$
Let $x\in \Omega$ and  $\bar{r}\le r\le  \delta_\gamma(\Omega)$.  It holds
\begin{equation}\label{r_big}
\begin{split}
& \frac{1}{|B_\gamma(x,r)\cap \Omega|^\lambda}     \int_{B_\gamma(x,r)\cap \Omega}|f(y)\varphi(y)-(f\varphi)_{B_\gamma(x,r)\cap \Omega}|^p \\
& \le \frac{2^{p}\left\|\varphi\right\|^p_{L^{\infty}(K)}}{|B_\gamma(x,\bar{r})\cap \Omega|^{\lambda}}\left\|f\right\|^p_{L^{p}(K)} \le \frac{2^{p}\left\|\varphi\right\|^p_{L^{\infty}(K)}}{{\mathcal{I}(\bar{r})}^{\lambda}} \left\|f\right\|^p_{L^{p}(K)},
\end{split}
\end{equation}
where $\mathcal{I}(\bar{r})$ is defined in \eqref{inf_ball}.\\
We can therefore restrict our attention to values of $r$ with  $0<r<\bar{r}$. Note that if a ball $B_\gamma$ is such that $B_\gamma\cap K=\emptyset$ we trivially have 
$$ \frac{1}{|B_\gamma\cap \Omega|^\lambda}  \int_{B_\gamma\cap \Omega}|f(y)\varphi(y)-(f\varphi)_{B_\gamma\cap \Omega}|^p =0.$$ 
In particular, if we set $\widetilde U:=\cup_{x\in K}B_{\gamma}(x,2\bar r)$, we see that  $B_\gamma(x,r)\cap K=\emptyset$   if $x\in \Omega\setminus \widetilde U$.   We also note that   if $x\in \widetilde U$ then $B_\gamma(x,r)\subset U$.   For these reasons,  it suffices to control the quantity 
\begin{equation}
\sup_{x\in \widetilde U}\sup_{r\in ]0,\bar{r}[}\frac{1}{|B_\gamma(x,r)|^\lambda}\int_{B_\gamma(x,r)}|f(y)\varphi(y)-(f\varphi)_{B_\gamma(x,r)}|^{p}.
\end{equation}
Since $B_\gamma(x,r)\subset U$ whenever $x\in\widetilde U$, we have that 
\begin{equation}\label{BMO_campanato}
\begin{split}
&\sup_{x\in \widetilde{U}}\sup_{r\in ]0,\bar{r}[}\frac{1}{|B_\gamma(x,r)|^\lambda}\int_{B_\gamma(x,r)}|f(y)\varphi(y)-(f\varphi)_{B_\gamma(x,r)}|^{p} \\
&\le \sup_{x\in U} \sup_{r\in]0,\delta_{\gamma}(U)]}\frac{1}{|B_\gamma(x,r)\cap U|^\lambda}\int_{B_\gamma(x,r)\cap U}|f(y)\varphi(y)-(f\varphi)_{B_\gamma(x,r)\cap U}|^{p},
\end{split}
\end{equation}
hence it suffices  to estimate the quantity $|f\varphi|^{p}_{\mathcal{L}^{\lambda}_{p,\gamma}(U)}$.  
We analyse first the case $\lambda=1$.
Keeping in mind  Lemma \ref{normeeq}, we estimate the appropriate integral from \eqref{eqsem} as follows:
\begin{equation}\label{split1}
\dashint_{B_\gamma\cap U }\dashint_{B_\gamma\cap U  }|f(y)\varphi(y)-f(z)\varphi(z)|^pdydz \le  2^{p-1}(I_1+I_2)
\end{equation}
where  we have set 
\begin{equation}\label{split2}
\begin{split}
& I_1=\dashint_{B_\gamma\cap U }\dashint_{B_\gamma\cap U }|\varphi(y)|^p |f(y)-f(z)|^pdydz\\
& I_2=\dashint_{B_\gamma\cap U }\dashint_{B_\gamma\cap U}|f(z)|^p |\varphi(y)-\varphi(z)|^p dydz,
\end{split}
\end{equation}
and for simplicity we have denoted by $B_{\gamma}$ a generic ball with center in $U$ and radius non exceeding the diameter of $U$. 
 Clearly, we have 
\begin{equation}\label{prox}I_1\le ||\varphi||^p_{L^{\infty}(U ) }  |f|^p_{ \mathcal{L}^{1}_{p,\gamma}(U )}
.\end{equation}
Now let $q>\max\left\{1,\frac{\gamma N_\gamma}{p}\right\}$ be fixed. Observe that by Theorem~\ref{daprato}, function $f$ belongs to any $L^r(U)$  (in particular to $L^{pq}(U)$) and its $L^r$-norm is estimated by  its Campanato norm.   By combining H\"older, Jensen inequality, Lemma~\ref{normeeq}  and Theorem~\ref{daprato}, we get that 
\begin{equation}\label{prox1}
\begin{split}
 I_2 & \le  \frac{1}{|B_\gamma\cap U|}\left\|f\right\|^p_{L^{pq}(B_\gamma\cap U)}\left(\int_{B_\gamma\cap U}\left(\dashint_{B_\gamma\cap U}|\varphi(y)-\varphi(z)|^pdy \right)^{q'} dz \right)^{\frac{1}{q'}}\\
& \le\left\|f\right\|^p_{L^{pq}(B_\gamma\cap U)}\left( |B_{\gamma}\cap U|^{ 1-q' }      \dashint_{B_\gamma\cap U}\dashint_{B_\gamma\cap U}|\varphi(y)-\varphi(z)|^{pq'}  dydz\right)^{\frac{1}{q'}}\\
& \le C \|f\|^p_{L^{pq}(U)}  |\varphi|^p_{\mathcal{L}^{q'}_{pq',\gamma}(U)}  \le C   \|f\|^p_{ \mathcal{L}^{1}_{p,\gamma}(U )}
\|\varphi\|^p_{C^{0,\beta}(\bar{U},\delta_\gamma)}
\\
\end{split}
\end{equation}
where $q'=q/(q-1) $ and 
$\beta=\frac{N_{\gamma} }{pq}$ (note that $\gamma \beta <1$  by our choice of $q$ and $\|\varphi\|_{C^{0,\beta}(\bar{U},\delta_\gamma)}$ is finite).

By \eqref{r_big}, \eqref{split1},   \eqref{prox}, \eqref{prox1}  and Lemma~\ref{normeeq}, we conclude that 
$$| f\varphi |_{ \mathcal{L}^{1}_{p,\gamma}(\Omega )} \le C\left(
\|f\|_{ \mathcal{L}^1_{p,\gamma}(\Omega )}+\|f\|_{L^{p}(U)}\right).
$$
and the proof is complete for $\lambda=1$. \\
If $\lambda\neq 1$,  by Theorem \ref{daprato}, 
 the seminorm in \eqref{BMO_campanato} is controlled  (up to a positive constant not depending on $f$) by the corresponding Morrey or H\"older norm in $U$,  $\|f\varphi\|^p_{L^{\lambda}_{p,\gamma}(U)}$ and $\|f\varphi\|^p_{C^{0,\alpha}(\overline{U}, \delta_\gamma)}$ respectively.  Moreover,  by the definitions of the latter norms,  Theorem~\ref{daprato} and Lemma~\ref{restrizioneholder}  below we have
\begin{equation*}
\begin{split}
\|f\varphi\|_{L^{\lambda}_{p,\gamma}(U)} \le ||\varphi||_{L^{\infty}(U)}\|f\|_{L^{\lambda}_{p,\gamma}(U)} \le C\|f\|_{\mathcal{L}^{\lambda}_{p,\gamma}(U)} \le  \widetilde{C}\|f\|_{\mathcal{L}^{\lambda}_{p,\gamma}(\Omega)},
\end{split}
\end{equation*}
and similarly 
$\|f\varphi\|_{C^{0,\alpha}(\overline{U}, \delta_\gamma)} \le\widetilde{C}\|f\|_{\mathcal{L}^{\lambda}_{p,\gamma}(\Omega)}$ (here we use  the fact that $\varphi \in C^{0,\alpha}(\overline{U}, \delta_\gamma)$ since $\gamma \alpha \le1$ by assumption and $\varphi$ is smooth). 
Thus \eqref{product_final} follows also for $\lambda\ne1$. 
\end{proof}

The following analogous result can be useful

\begin{lemma}\label{productbis}  
Let $p\in [1,\infty[  $, $\gamma\in ]0,1]$, $\lambda \in ]0,\infty [$.
  If $\lambda >1$ assume also  that $\gamma \alpha \le 1 $ where $\alpha $ is defined in Theorem~\ref{daprato}. Let $\Omega$ be a bounded domain  in $\R^N$ satisfying \eqref{lemmacusp1} for all $x\in \overline{\Omega}$ and some positive constant $c$.
Let  $f\in \mathcal{L}^{\lambda}_{p,\gamma}(\Omega)$ and $\varphi\in  C_c^{\infty}(\R^N)$.  
Then there exists a positive constant $C$  depending  only on  $N, p,\lambda , \gamma, \Omega, c $ and $\varphi$ such that 
\begin{equation}\label{product_final0}
|f\varphi|_{ \mathcal{L}^{\lambda}_{p,\gamma}(\Omega)}\le C \| f\|_{ \mathcal{L}^{\lambda}_{p,\gamma}(\Omega)}.
\end{equation}
\end{lemma}

\begin{proof} By Theorem~\ref{daprato} it suffices to prove the lemma only in the case $\lambda =1$  since for $\lambda \ne 1$ the norms under consideration are equivalent to Morrey and H\"{o}lder's norms for which the statement is straightforward. Moreover, again by Theorem~\ref{daprato} it suffices to consider the case $p=1$. Then the proof can be easily done  by  following the lines from \eqref{split1} to \eqref{prox1} in the proof of the previous lemma, with $U$ replaced by $\Omega$. 
\end{proof}

\subsection{Restrictions}

The following easy  lemma  will be used in the sequel.

\begin{lemma}\label{restrizioneholder}
Let $\Omega$ be a domain  in $\mathbb\R^N$ such that inequality \eqref{lemmacusp1} holds for all $x\in \overline{\Omega}$ and a positive constant $c$ independent of $x$.  If 
$f$ belongs to the space $\mathcal{L}^{\lambda}_{p,\gamma}(\widetilde{\Omega})$ for some domain $\widetilde{\Omega}$ containing $\Omega$,  then its restriction $f_{|\Omega}$ to $\Omega$ belongs to $\mathcal{L}^{\lambda}_{p,\gamma}(\Omega)$ and there exists a positive constant $C$ depending only on $N,p,\lambda$ and $c$ such that 
$$
\left|f_{|\Omega}\right|_{ \mathcal{L}^{\lambda}_{p,\gamma}(\Omega)  } \le C \left| f \right|_{\mathcal{L}^{\lambda}_{p,\gamma}(\widetilde{\Omega})}\, .
$$
\end{lemma}
\begin{proof}
For any ball $B_{\gamma}$ with center in $\Omega$ and radius not exceeding the diameter of $\Omega$ we have
\begin{equation}\label{eq_restriction}
\begin{split}
 & |B_\gamma \cap \Omega|^{1-\lambda}\dashint_{B_\gamma \cap\Omega}\dashint_{B_\gamma \cap\Omega}\left|f(y)-f(z)\right|^pdydz \\
& \le  |B_\gamma \cap \Omega|^{1-\lambda}\frac{1}{c^2r^{2N_{\gamma}} }   \int_{B_\gamma\cap \widetilde{\Omega}}\int_{B_\gamma\cap \widetilde{\Omega}}\left|f(y)-f(z)\right|^p dydz   \\
& \le  |B_\gamma\cap \Omega|^{1-\lambda} \frac{C}{ \left|B_\gamma\right|^2}   \int_{B_\gamma\cap \widetilde{\Omega}}\int_{B_\gamma\cap \widetilde{\Omega}}\left|f(y)-f(z)\right|^p dydz   \\
& 
\le C |B_\gamma\cap \Omega|^{1-\lambda} \dashint_{B_\gamma\cap \widetilde{\Omega}} \dashint_{B_\gamma\cap \widetilde{\Omega}}\left|f(y)-f(z)\right|^p dydz  ,
\end{split}
\end{equation}
for some positive constant $C$.  We distinguish two cases. 
If $\lambda\le 1$  from \eqref{eq_restriction} we immediatly get 
that 
\begin{equation}\label{eq_restriction_bis}
\begin{split}
 & |B_\gamma\cap \Omega|^{1-\lambda}\dashint_{B_\gamma\cap\Omega}\dashint_{B_\gamma\cap\Omega}\left|f(y)-f(z)\right|^pdydz \\
& \le  |B_\gamma\cap \widetilde{\Omega}|^{1-\lambda}\dashint_{B_\gamma\cap \widetilde{\Omega}}\dashint_{B_\gamma\cap \widetilde{\Omega}}\left|f(y)-f(z)\right|^p dydz. 
\end{split}
\end{equation}
If $\lambda>1$,  again by \eqref{lemmacusp1}
\begin{equation}\label{eq_restriction_tris}
\begin{split}
 & |B_\gamma\cap \Omega|^{1-\lambda}\dashint_{B_\gamma\cap\Omega}\dashint_{B_\gamma\cap\Omega}\left|f(y)-f(z)\right|^pdydz \\
 & \le (c r^{N_{\gamma}})^{-1-\lambda}\int_{B_\gamma\cap\Omega}\int_{B_\gamma\cap\Omega}\left|f(y)-f(z)\right|^pdydz \\
 &\le C  | B_\gamma |^{-1-\lambda}\int_{B_\gamma\cap\Omega}\int_{B_\gamma\cap\Omega}\left|f(y)-f(z)\right|^pdydz \\
& \le  C |B_\gamma\cap \widetilde{\Omega}|^{-1-\lambda} \int_{B_\gamma\cap \widetilde{\Omega}}\int_{B_\gamma\cap \widetilde{\Omega}}\left|f(y)-f(z)\right|^p dydz\\
& \le  C |B_\gamma\cap \widetilde{\Omega}|^{1-\lambda} \dashint_{B_\gamma\cap \widetilde{\Omega}}\dashint_{B_\gamma\cap \widetilde{\Omega}}\left|f(y)-f(z)\right|^p dydz,
\end{split}
\end{equation}
for some positive constant $C$. The thesis follows from Lemma \ref{normeeq}.
\end{proof}

\section{Extension theorems }

In this section we prove a number of extension results starting with some observations concerning the extension by zero.

\subsection{Extension by zero}

The extension by zero requires some attention in the case of $\bmo$ spaces, as the following classical example shows. 

\example \label{log_0}  Let $\Omega =\left\{x\in \R: x> 0 \right\}.$ It is known and easy to see that $f(x)=\log x \in \bmo(\Omega )$. However,   $f_0\notin \bmo(\R)$. Indeed,  let's consider the family of intervals  centred at zero with radius $r$.  Then,  
$$\frac{1}{2r}\int_{-r}^{r} f_0(x)dx=\frac{1}{2r}\int_{0}^{r}\log x\, dx     =\frac{\log r-1}{2}. $$
Moreover, it's clear that 
$$\frac{1}{2r}\int_{-r}^{r}\left|f_0(x)-\frac{\log r-1}{2}\right|dx\ge  \frac{1}{2r}\int_{-r}^{0}\left|f_0(x)-\frac{\log r-1}{2}\right|dx =     \left|\frac{\log r -1}{4}\right| $$
which diverges as $r$ goes to zero and  to infinity. \\

The previous example also points out   that taking balls $B$ in \eqref{clasintro}  with center in $\Omega$ is not equivalent to taking all balls, possibly centred outside $\Omega$.   
 
 If the (essential) support ${\rm supp } f$ of $f$ is compact  then the extension by zero  of $f$ is admissible. 
 
\begin{lemma} \label{compactlem}Let $\Omega$ be a domain in $\R^N$ and $f\in  \mathcal{L}^{\lambda}_{p,\gamma}(\Omega) $. If ${\rm supp } f$ is compact in $\Omega$ then 
$f_0\in  \mathcal{L}^{\lambda}_{p,\gamma}(\R^N)$ and there exists a positive constant $C$ depending only on $N,p, \lambda, \gamma$ and   $\delta_{\gamma}( {\rm supp } f, \partial \Omega  )$  such that 
$$
| f_0 |_{ \mathcal{L}^{\lambda}_{p,\gamma}(\R^N)}\le C \| f \|_{ \mathcal{L}^{\lambda}_{p,\gamma}(\Omega)}\, .
$$
\end{lemma}
 
 \begin{proof}
Let $\bar r\in ]0, \delta_{\gamma}({\rm supp } f, \partial \Omega)  /4[$   be fixed. If $r\geq \bar r$ then for any ball $B_{\gamma}$ with radius $r$ we have
\begin{equation}\label{r_big_bis}
 \frac{1}{|B_\gamma|^\lambda}     \int_{B_\gamma }|f(y)-f_{B_\gamma }|^pdy 
 \le \frac{2^{p}\left\| f\right\|^p_{L^{p}(\Omega)}}{  (2\omega_{N-1}\bar r^{N_{\gamma}}  )   ^{\lambda}} \, .
\end{equation}
 Assume now that  $0<r<\bar r $. Let  $B_{\gamma}$ be a ball  with radius $r$. Let $U=\cup_{x\in {\rm supp} f}B_{\gamma}(x,2\bar r) $.  If  $B_{\gamma}$  is centred outside $U$ then $ \int_{B_\gamma }|f(y)-f_{B_\gamma }|^pdy =0$. If  $B_{\gamma}$ is centred in $U$ then $B_{\gamma}\subset \Omega $ hence 
$$  \frac{1}{|B_\gamma|^\lambda}     \int_{B_\gamma }|f(y)-f_{B_\gamma }|^pdy \le |f|_{ \mathcal{L}^{\lambda}_{p,\gamma}(\Omega) }
$$
and the proof is complete. 
\end{proof}
\indent The situation is simpler if we deal with Morrey spaces.  The following result indeed shows that in this case the extension by zero is always  possible.
\begin{lemma}\label{extensionbyzero}
Let $\Omega$ be a domain in $\R^N$ and $f\in L^{\lambda}_{p,\gamma}(\Omega)$.
 Then $f_0\in L^{\lambda}_{p,\gamma}(\R^N)$ and  there exists a positive constant $C$ independent of $f$ such that 
\begin{equation}\label{morrey00}\|f_0\|_{L^{\lambda}_{p,\gamma}(\R^N)}\le C\|f\|_{L^{\lambda}_{p,\gamma}(\Omega)}. 
\end{equation}
\end{lemma}
\begin{proof}
Assume first that  $\Omega$ is unbounded. 
Let $B_\gamma=B_\gamma(x,r) $ be a fixed ball in $\R^N$ with radius $r\in ]0,\infty [$. If $x\in \Omega$ then we immediately have that 
\begin{equation}\label{morrey01}
\begin{split}
\frac{1}{|B_\gamma|^{\lambda}}\int_{B_\gamma}|f_0(y)|^{p} dy&  =\frac{1}{|B_\gamma|^{\lambda}}\int_{B_\gamma\cap \Omega}|f(y)|^{p} dy \\ &  \le \frac{1}{|B_\gamma\cap \Omega|^{\lambda}}\int_{B_\gamma\cap \Omega}|f(y)|^{p} dy
 \le \|f\|^p_{L^{\lambda}_{p,\gamma}(\Omega)}.
\end{split}
\end{equation}
Assume now that $x\in \R^N\setminus \Omega$. The analysis of the case  $B_{\gamma}\cap \Omega= \emptyset$ is trivial since in this case  $f_0$ vanishes on $B_{\gamma}$. Thus, we assume that $B_{\gamma}\cap \Omega \ne \emptyset$ and we fix $z\in B_{\gamma}\cap \Omega$. We note that 
$B_{\gamma }\subset B_{\gamma}(z,2r) $ and that  $|B_{\gamma}(z,2r) | = 2^{N_{\gamma}} |B_{\gamma}| $, hence 

\begin{equation}\label{morrey02}
\begin{split}
\frac{1}{|B_\gamma|^{\lambda}}\int_{B_\gamma}|f_0(y)|^{p} dy  =\frac{1}{|B_\gamma|^{\lambda}}\int_{B_\gamma\cap \Omega}|f(y)|^{p} dy = \frac{2^{\lambda N_\gamma}}{|B_\gamma(z,2r)|^{\lambda}}\int_{B_\gamma \cap \Omega}|f(y)|^pdy\\
 \le \frac{2^{\lambda N_\gamma}}{|B_\gamma(z,2r)\cap \Omega|^{\lambda}}\int_{B_\gamma(z,2r)\cap \Omega}|f(y)|^pdy\le 2^{\lambda N_\gamma} \|f\|^p_{L^{\lambda}_{p,\gamma}(\Omega)}.
\end{split}
\end{equation}
By \eqref{morrey01} and \eqref{morrey02} we deduce the validity of \eqref{morrey00} when $\Omega$ is unbounded.

Assume now that $\Omega$ is bounded. In this case if the radius $r$ of the ball $B_{\gamma}(x,r)$ satisfies the condition $r\le \delta_{\gamma}(\Omega )   /2$, then inequalities \eqref{morrey01} and \eqref{morrey02} holds because the balls of radii $r$ and $2r$  are  admissible in the definition  of the Morrey norm in $\Omega$. For  balls $B_{\gamma}(x,r)$ with  radius $r> \delta_{\gamma }(\Omega)$ the estimate of the ratio in \eqref{normmorrey} is straightforward since 
the measure at the denominator is uniformly bounded away from  zero and the $L^p$-norm at the numerator is estimated by the norm in $L^p(\Omega)$ of $f$ which, in turn, is estimated by the Morrey norm (recall that now we are considering the case when $\Omega$ is bounded).  
\end{proof}

\subsection{Extension for elementary $C^{0,\gamma}$ domains}

In this section we consider  elementary domains  of class $C^{0,\gamma}$, which means that they are  defined as subgraphs of  H\"older continuous  functions   with exponent $\gamma\in ]0,1]$. Namely, we consider  open sets    $\Omega$  in ${\mathbb R}^N$, $N\geq 2$,   in the form 
  \begin{equation}\label{eldom}\Omega = \{x=(\overline x, x_N)\in {\mathbb R}^N:\ \bar x\in \R^{N-1}\  x_N<\varphi(\overline x) \}\,,
\end{equation}
where $\varphi :  \R^{N-1}\to {\mathbb R}$  is a function satisfying the condition 
\begin{equation}\label{lip1}
|\varphi(\overline x)-\varphi(\overline y)| \le M|\overline x-\overline y|^{\gamma} 
\end{equation}
for all $ \overline x, \overline y\in {\mathbb R}^{N-1}$ and 
 some positive constant $M$. 
The best constant $M$ in inequality (\ref{lip1}) is denoted by ${\rm Lip}_{\gamma} \varphi$.
Recall that the case $\gamma=1$ corresponds to the Lipschitz case.  

We recall that an elementary open set as above satisfies the cusp condition which is a generalisation of the usual cone condition. Namely, a cusp  with exponent $\gamma$, 
 vertex $x\in \R^N$,  and opening $M$ is defined by 
\begin{equation} 
C_{\gamma}(x, M)=\{y\in {\mathbb R}^N:\    y_N<  x_N-M|\bar y -\bar x|^{\gamma}  \}
\end{equation}
and it turns out  that if the function $\varphi$ defining $\Omega$ in \eqref{eldom} satisfies \eqref{lip1} then 
$C_{\gamma}(x, M )\subset  \Omega $ for all $x\in \overline\Omega$.  This implies the existence of a constant $c>0$ depending only on   $N,\gamma $ and $M$ 
such that  inequality \eqref{lemmacusp1} holds 
for all $x\in \overline{\Omega }$ and $r>0$. For a proof we refer to \cite{lamves}.  

Then we can prove the following extension result. 

\begin{teorema}\label{estensioneholder}  Let $\Omega$ be an elementary  domain in $\mathbb{R}^N$  of class $C^{0,\gamma}$ with  $\gamma\in ]0,1]$, as in \eqref{eldom}. 
Given $f\in \mathcal{L}^{\lambda}_{p,\gamma}(\Omega)$, the function $\tilde{f}$ defined by
$$\tilde{f}(\bar{x},x_N)=\begin{cases} f(\bar{x},x_N), & \mbox{if}\ x_N<\varphi(\bar{x}), \\ f(\bar{x},2\varphi(\bar{x})-x_N), & \mbox{if}\ x_N>\varphi(\bar{x})
\end{cases}$$
belongs to  $\mathcal{L}^{\lambda}_{p,\gamma}(\mathbb{R}^N)$ and there exits a positive constant $C$ depending only on $N, p, \lambda, \gamma$ and $M$ such that 
\begin{equation}\label{mainest}
|\tilde{f}|_{\mathcal{L}^{\lambda}_{p,\gamma}(\mathbb{R}^N)}\le  C |f|_{\mathcal{L}^{\lambda}_{p,\gamma}(\Omega)}.
\end{equation}
\end{teorema}
\begin{proof}  Let $B_{\gamma}(x,r)$ be a generic ball in $\R^N$ centred at  $x=(\bar{x},x_N)$ with radius $r$. We plan to estimate  the  double integral
$$\dashint_{B_{\gamma}(x,r)}\dashint_{B_{\gamma}(x,r)}|\tilde f(y)-\tilde f(z)|^p dydz.$$
We distinguish two cases:\\
\indent Case 1.  In this case we assume that  $B_{\gamma}(x,r)\subset\overline{\Omega^c}$. Let  $$D=\left\{(\bar{w},2\varphi(\bar{w})-w_N): (\bar{w},w_N)\in B_\gamma(x,r)\right\}$$ be the symmetric set  of $B_\gamma$ with respect to the graph of the function $\varphi$. By the definition of $B_{\gamma}(x,r)$ we have that  for all $(\bar{w},w_N)\in B_{\gamma}(x,r)$  
$$|\bar{w}-\bar{x}|<r^{\frac{1}{\gamma}}$$
and $$ |2\varphi(\bar{w})-w_N-(2\varphi(\bar{x})-x_N)|\le |x_N-w_N|+2|\varphi(\bar{x})-\varphi(\bar{w})| $$
$$\le r+2M|\bar{x}-\bar{w}|^{\gamma}<(1+2M)r.$$
 Hence 
 \begin{equation}
 D\subset B_{\gamma}(  \tilde x  ,\tilde{r})
 \end{equation}
 where $\tilde x=(\bar x, 2\varphi (\bar x)-x_N)$ and  $\tilde r= (1+2M)r$. 
 Note that 
 \begin{equation}\label{r_tilde}
|B_{\gamma}(\tilde x,\tilde{r})|=2\omega_{N-1}\left((1+2M)r\right)^{N_\gamma}=(1+2M)^{N_\gamma}|B_{\gamma}(x,r)| .
\end{equation}
By changing variables in integrals, we have 
\begin{multline} \label{r_tilde_due}
\dashint_{B_{\gamma}(x,r)}\dashint_{B_{\gamma}(x,r)}|   \tilde{f}(y)-\tilde{f}(z)|^pdydz=\frac{1}{|B_{\gamma}(x,r)|^2}\int_{D}\int_{D}|f(y)-f(z)|^pdydz \\
=\frac{(1+2M)^{2N_\gamma}}{|B_{\gamma}(\tilde x,\tilde{r})|^2}\int_{D}\int_{D}|f(y)-f(z)|^pdydz \\
\le (1+2M)^{2N_\gamma}\dashint_{B_{\gamma}(\tilde x,\tilde{r})\cap\Omega   }\dashint_{B_{\gamma}(x,\tilde{r})\cap\Omega   }|f(y)-f(z)|^pdydz .
\end{multline}
Thus,  by \eqref{r_tilde} and \eqref{r_tilde_due}
\begin{multline}\label{case1}
|B_\gamma(x,r)|^{1-\lambda}\dashint_{B_{\gamma}(x,r)}\dashint_{B_{\gamma}(x,r)}|   \tilde{f}(y)-\tilde{f}(z)|^pdydz \\
 \le  |B_\gamma(\tilde{x},\tilde{r})|^{1-\lambda}(1+2M)^{N_\gamma(1+\lambda)}
 \dashint_{B_{\gamma}(\tilde x,\tilde{r})\cap\Omega   }\dashint_{B_{\gamma}(x,\tilde{r})\cap\Omega   }|f(y)-f(z)|^pdydz.
\end{multline}
In particular,  if $\lambda\le 1$ we use \eqref{lemmacusp1} to get 
\begin{equation}\label{case12}
\begin{split}
& |B_\gamma(x,r)|^{1-\lambda}\dashint_{B_{\gamma}(x,r)}\dashint_{B_{\gamma}(x,r)}|   \tilde{f}(y)-\tilde{f}(z)|^pdydz\\
& \le C(1+2M)^{N_\gamma(1+\lambda)}|B_\gamma(\tilde{x},\tilde{r})\cap \Omega|^{1-\lambda}
 \dashint_{B_{\gamma}(\tilde x,\tilde{r})\cap\Omega   }\dashint_{B_{\gamma}(x,\tilde{r})\cap\Omega   }|f(y)-f(z)|^pdydz\\
& \le C 2^{p}(1+2M)^{N_\gamma(1+\lambda)}|f|^p_{\mathcal{L}^{\lambda}_{p,\gamma}(\Omega)}.
\end{split}
\end{equation}
while if $\lambda>1$ we obtain 
\begin{eqnarray}\label{case13}\lefteqn{
|B_\gamma(x,r)|^{1-\lambda}\dashint_{B_{\gamma}(x,r)}\dashint_{B_{\gamma}(x,r)}|   \tilde{f}(y)-\tilde{f}(z)|^pdydz} \nonumber \\
& & 
\le 
|B_\gamma(x,r) \cap \Omega|^{1-\lambda}\dashint_{B_{\gamma}(x,r)}\dashint_{B_{\gamma}(x,r)}|   \tilde{f}(y)-\tilde{f}(z)|^pdydz\nonumber 
\\ & & 
\le 2^{p}(1+2M)^{N_\gamma(1+\lambda)}|f|^p_{\mathcal{L}^{\lambda}_{p,\gamma}(\Omega)}.
\end{eqnarray}
\indent Case 2. In this case we assume that  $B_{\gamma}(x,r)\cap\Omega\neq \emptyset$.  To shorten our notation, we set  $B_\gamma=B_{\gamma}(x,r)$. 
We write
\begin{multline} 
\frac{1}{|B_\gamma|^2}\int_{B_\gamma}\int_{B_\gamma}|\tilde{f}(y)-\tilde{f}(z)|^pdydz \\ =\frac{1}{|B_\gamma|^2}\int_{B_{\gamma}\cap\Omega}\int_{B_{\gamma}\cap\Omega}|f(y)-f(z)|^pdydz
+\frac{2}{|B_\gamma|^2}\int_{B_{\gamma}\cap\Omega}\int_{B_{\gamma}\cap\Omega^c}|\tilde{f}(y)-f(z)|^pdydz \\
+\frac{1}{|B_\gamma|^2}\int_{B_{\gamma}\cap\Omega^c}\int_{B_{\gamma}\cap\Omega^c}|\tilde{f}(y)-\tilde{f}(z)|^pdydz
=I_1+I_2+I_3
\end{multline}
where $$I_1=\frac{1}{|B_\gamma|^2}\int_{B_{\gamma}\cap\Omega}\int_{B_{\gamma}\cap\Omega}|f(y)-f(z)|^pdydz,$$
$$I_2= \frac{2}{|B_\gamma|^2}\int_{B_{\gamma}\cap\Omega}\int_{B_{\gamma}\cap\Omega^c}|\tilde{f}(y)-f(z)|^pdydz$$
and $$I_3=\frac{1}{|B_\gamma|^2}\int_{B_{\gamma}\cap\Omega^c}\int_{B_{\gamma}\cap\Omega^c}|\tilde{f}(y)-\tilde{f}(z)|^pdydz.$$

We look for a ball centred at some point of $\Omega$ containing both the symmetric set  of  $B_\gamma\cap\Omega^c$ with respect to the graph of $\varphi$ and $B_\gamma\cap\Omega$.
We can clearly restrict our attention to  the case $B_\gamma\cap \Omega^c \neq \emptyset$ otherwise the conclusion is trivial. Thus,  we fix a point  $(\bar{u},u_N)\in B_\gamma\cap\Omega^c$. Since $B_{\gamma}\cap \Omega \ne \emptyset$, we can choose $(\bar u, u_N)$ satisfying the condition
\begin{equation}\label{funzionale}
(\bar u, \varphi (\bar u))\in B_{\gamma}\, .
\end{equation}

Then we note that for all  $(\bar{w},w_N)\in B_\gamma\cap\Omega^c$  we have
$ |\bar{w}-\bar{u}|\le |\bar w- \bar x  |+|\bar x - \bar u|  < 2r^{\frac{1}{\gamma}}$
and
\begin{equation} |2\varphi(\bar{w})-w_N-2\varphi(\bar{u})+u_N|\le 2M|\bar{w}-\bar{u}|^{\gamma}+2r 
\le \bar r,
\end{equation}
where we have set $\bar r= 2^{\gamma+1}Mr+2r$.
Thus  $  B_\gamma(\left(\bar{u},2\varphi(\bar{u})-u_N),\bar{r}\right)\cap \Omega $  contains the symmetric of $B_{\gamma}\cap\Omega^c$. 
Moreover, for all $(\bar{w},w_N)\in B_{\gamma}\cap\Omega$ we have 
$|\bar{u}-\bar{w}|< 2r^{\frac{1}{\gamma}}$ and 
\begin{multline}
|2\varphi(\bar{u})-u_N-w_N|\le |\varphi(\bar{u})-u_N|+|\varphi(\bar{u})-w_N|\\
\le 2r+|\varphi(\bar{u})-\varphi(\bar{w})|+|\varphi(\bar{w})-w_N|\\
\le 4r+M|\bar{u}-\bar{w}|^{\gamma}\le    4r+2^{\gamma}Mr   \le 2\bar r,
\end{multline}
where we have used also condition \eqref{funzionale}. 
Therefore, $B_{\gamma}((\bar{u},2\varphi(\bar{u})-u_N),2\bar{r}))$ contains 
$B_{\gamma}\cap\Omega$.  Thus we set $\widetilde{B}_\gamma= B_{\gamma}((\bar{u},2\varphi(\bar{u})-u_N),2\bar{r}))$
and we observe that 

\begin{equation}\label{r_bar}
|\widetilde{B}_\gamma|=C  |B_\gamma |,\quad C=(4(1+2^{\gamma}M))^{N_{\gamma }}.
\end{equation}
By changing variables in integrals we have that 
\begin{multline}
I_1+I_2+I_3\le \frac{4}  {|{B}_\gamma|^2}\int_{\widetilde{B}_{\gamma}\cap\Omega}\int_{\widetilde{B}_{\gamma}\cap\Omega}|f(y)-f(z)|^pdydz\\
\le \frac{4C^2}{|\widetilde{B}_{\gamma}|^2}\int_{\widetilde{B}_{\gamma}\cap\Omega}\!\int_{\widetilde{B}_{\gamma}\cap\Omega}\!\!  |f(y)-f(z)|^pdydz
\le  \frac{4C^2}{|\widetilde{B}_{\gamma} \cap \Omega  |^2}\int_{\widetilde{B}_{\gamma}\cap\Omega}\!   \int_{\widetilde{B}_{\gamma}\cap\Omega}\!\!   |f(y)-f(z)|^pdydz \\
= 4C^2 \dashint_{\widetilde{B}_{\gamma}\cap\Omega}\dashint_{\widetilde{B}_{\gamma}\cap\Omega}|f(y)-f(z)|^pdydz  .
\end{multline}
Thus, by \eqref{r_bar} and by arguing as in the proof of inequalities \eqref{case12}, \eqref{case13} from \eqref{case1},  we get
\begin{multline}\label{case2}
|B_\gamma|^{1-\lambda}\dashint_{B_\gamma}\dashint_{B_\gamma}|\tilde{f}(y)-\tilde{f}(z)|^pdydz \le 
4C^2 |B_\gamma|^{1-\lambda} \dashint_{\widetilde{B}_{\gamma}\cap\Omega}\dashint_{\widetilde{B}_{\gamma}\cap\Omega}|f(y)-f(z)|^pdydz \\
\le  4C^{1+\lambda}|\widetilde{B}_\gamma|^{1-\lambda}\dashint_{\widetilde{B}_{\gamma}\cap\Omega}\dashint_{\widetilde{B}_{\gamma}\cap\Omega}|f(y)-f(z)|^pdydz
\le D |f|^{p}_{\mathcal{L}^{\lambda}_{p,\gamma}(\Omega)}
\end{multline}
for some positive constant $D$. 
In conclusion,  by combining inequalities \eqref{case12}, \eqref{case13}, \eqref{case2} and Lemma~\ref{normeeq} we deduce the validity of estimate \eqref{mainest}. 
\end{proof}

\subsection{Extension for sum of Campanato spaces}

In  this section we consider open sets $\Omega$  which can be locally described  near the boundary  by suitable  subgraphs of $C^{0,\gamma}$ functions.   This means that for any point of the boundary there exists a neighborhood $U$ and an isometry $R$ such that  $R(\Omega \cap U)$ is  an elementary domain of the form 
\begin{equation}\label{rotation}
\{(\bar x, x_N)\in \R^N:\ \bar x\in \R^{N-1},\ x_N\le \varphi (\bar x), \ \bar x \in W    \}
\end{equation}
where $W$ is an open set in $\R^{N-1}$ and $\varphi$ a real-valued function of class $C^{0,\gamma}$ on $W$, which means that inequality \eqref{lip1} is satisfied for all $\bar x, \bar y \in W$.  See Figure~\ref{holder_patchy}. As we have explained before, in the case of an elementary domain of class $C^{0,\gamma}$ described by variables $(\bar x, x_N)$ as in  \eqref{rotation}, the natural metric in the definition of the Campanato spaces is adapted to the direction  $x_N$.  Thus, since we have to use more rotations $R$ to describe different portions of the boundary we are forced to look 
at functions defined in $\Omega$ as sums of functions which belong to Campanato spaces associated with different metrics depending on the rotations $R$.  
This strategy can be applied not only to a bounded domain as above but to any domain.

Namely,  given an open set $\Omega$ in $\R^N$, $p\in [1,\infty [$, $\lambda \in ]0, \infty [$, $\gamma \in ]0,1 ]$, and an isometry $R$  of ${\mathbb{R}}^N$, we consider the space
\begin{equation}\label{elleerre}
 {{\mathcal L}^{\lambda, R }_{p, \gamma }(\Omega) }  =    \{   f\in L^p (\Omega ):\ 
    | f |_ { {\mathcal L}^{\lambda, R }_{p, \gamma }(\Omega)} \ne \infty   \} 
\end{equation}
where the seminorm  $| f |_{{\mathcal L}^{\lambda,R }_{p, \gamma }(\Omega) } $ is defined by 
\begin{eqnarray}\label{camp_i_norm}\lefteqn{
| f |_{{\mathcal L}^{\lambda,R }_{p, \gamma }(\Omega) } :=\sup_{x\in R (\Omega)}\sup_{r\in ]0, \delta_{\gamma}(R(\Omega))[ }\left(\frac{1}{|B_{\gamma}(x, r)\cap R (\Omega)|^{\lambda}  }\int_{B_{\gamma}(x, r)\cap R (\Omega)}\bigg| f \circ R^{-1}(y) \right. } \nonumber \\  & &  \qquad\qquad\qquad  \qquad\qquad \qquad\qquad\qquad  \left.  
-\dashint_{B_{\gamma}(x, r)\cap R (\Omega)}   f\circ R^{-1}(z)dz  \bigg|^pdy \right)^{\frac{1}{p}}  .
\end{eqnarray}

Note that since in this section we are mainly interested in the case of functions with support contained in bounded sets, for simplicity we have directly assumed that functions in \eqref{elleerre} belong to $L^p(\Omega)$. 

Now,  given a finite set ${\mathcal{R}}=\{R_k  \}_{k=1}^s$ of isometries $R_k$, we consider the space 
\begin{equation}\label{set_patchy}
    \mathcal{L}^{\lambda ,\mathcal R}_{p, \gamma }(\Omega):=\left\{f\in L^{p}(\Omega): f=\sum_{k=1}^{s}f_k,\
     f_k \in   {{\mathcal L}^{\lambda, R_k }_{p, \gamma }(\Omega) }    
\right\} 
\end{equation}
endowed  with the norm 
 \begin{equation}\label{norma_banach}
 \|  f\|_{\mathcal{L}^{\lambda ,\mathcal R}_{p, \gamma }(\Omega) }
 =\inf_{f=\sum_{k=1}^sf_k}  \left\{ \sum_{k=1}^{s}
   \| f_k \|_{ {\mathcal L}^{\lambda, R_k }_{p, \gamma }(\Omega) }     \right\} 
 \end{equation}
 where $   \| f_k \|_{ {\mathcal L}^{\lambda, R_k }_{p, \gamma }(\Omega) }   =  \|f_k  \|_{L^p(\Omega)} + | f_k |_{ {\mathcal L}^{\lambda, R_k }_{p, \gamma }(\Omega) }      $.
 
The above definitions can be specialised in the case of general open sets of class  $C^{0,\gamma}$ defined below.  Here, by cuboid we mean the isometric image of an $N$-dimensional parallelepiped with sides parallel to the coordinate axes.  Moreover, given a set $C$ in ${\mathbb R}^N$ and $\delta >0$ we denote by $C_{\delta}$ the set $\{x\in C:{\rm dist} (x,\partial C)>\delta \}$.

\begin{definizione}\label{atlas} Let $\gamma\in ]0,1]$,  $\delta \in ]0,\infty [$,  $s\in {\mathbb N}$. Let ${\mathcal{V}}=\{V_{k}\}_{k=1}^s$ be a family of cuboids and  ${\mathcal{R}}=\{R_k \}_{k=1}^s$ a corresponding family of isometries in ${\mathbb{R}}^N$ such that   for every $k=1,\dots , s$  we have 
$$
R_k( V_k  )= \Pi_{i=1}^N]a_{i,k}, b_{i,k}[
$$ 
where $a_{N,k}<a_{N,k}+\delta <  b_{N,k}-\delta <  b_{N,k}$.  We say that ${\mathcal{A} }=(s,\delta , {\mathcal{V}}, {\mathcal{R}} )$ is an {\em atlas}.  

Let $M\in  [0, \infty [$. We say that an open set $\Omega$ in ${\mathbb R}^N$ is of class $C^{0,\gamma}_M({\mathcal{A}})$  if the following conditions are satisfied:\\
 
(i) For every  $k=1,\dots , s$, we have $\Omega \cap (V_k)_{\delta }\ne \emptyset$. \\

(ii) $\Omega \subset \cup_{k=1}^{s}(V_k)_{\delta }$.\\

(iii) For every  $k=1,\dots ,s$, the set $\Omega_k:=\Omega \cap V_j$  satisfies the following condition: either $R_k(\Omega_k)= \Pi_{i=1}^N]a_{i,k}, b_{i,k}[ $ (in which case $V_k\subset \Omega $),  or $R_k(\Omega_k)$ is a bounded  elementary  domain of class $C^{0,\gamma}$
  of the form
\begin{equation*}\label{ele1bis}
R_k(\Omega_k)=\left\{x\in {\mathbb R}^N:\ \bar x\in W_k,\ a_{N,k}<x_N<\varphi_k(\bar x) \right\}
\end{equation*}
where  $\varphi_k$ is a real-valued   function   defined on $W_k= \Pi_{i=1}^{N-1}]a_{i,k}, b_{i,k}[$ such that  \eqref{lip1} is satisfied for all $\bar x, \bar y\in W_k$ and 
$$
a_{N,k}+\delta<\varphi_k<     b_{N,k}-\delta  
$$ 
(in which case $V_k\cap \partial \Omega \ne \emptyset$).   We say that $\Omega_k$ is a {\em patch} of $\Omega$.

Finally, we say that an open set $\Omega$ in ${\mathbb R}^N$ is of class $C^{0,\gamma} ({\mathcal{A}})$ if it is of class  $C^{0,\gamma}_M({\mathcal{A}})$ for some $M$.
\end{definizione}


Let $\Omega$ be a bounded domain of class  $C^{0,\gamma}(\mathcal{A})$. Assume that $\Omega_k$ is a patch of $\Omega$ as in Definition~\ref{atlas}.  We consider the subspace 
 $$ { \widetilde{ {\mathcal  L} }}  ^{\lambda, R_k }_{p, \gamma }(\Omega)  $$ of  ${{\mathcal L}^{\lambda, R_k }_{p, \gamma }(\Omega) } $  defined by  those functions $f$ in  $ {{\mathcal L}^{\lambda, R_k }_{p, \gamma }(\Omega) } $ such that 
 \begin{equation}\label{support}
 {\rm supp} f \subset  \overline{( V_k )_{\delta /2}}  .
 \end{equation}
 Note that we are not requiring that the support of $f$ has positive distance from $\partial \Omega$. Finally, we consider the subspace $  \mathcal{\widetilde{L}}^{\lambda ,\mathcal R}_{p, \gamma }(\Omega)$ of $ \mathcal{L}^{\lambda ,\mathcal R}_{p, \gamma }(\Omega)$ defined by
 \begin{equation}\label{patchyspace}
 \left\{f\in L^{p}(\Omega): f=\sum_{k=1}^{s}f_k,\
     f_k \in   {\widetilde {\mathcal {L}}^{\lambda, R_k }_{p, \gamma  }(\Omega) }  \, 
\right\} .
\end{equation}

\indent Then we can prove the following

\begin{lemma}\label{ext_elementary_hold}   
Let $\Omega$ be a domain  in $\R^N$ of class $C^{0,\gamma}_M(\mathcal{A})$ where  ${\mathcal{A} }=(s,\delta , {\mathcal{V}}, {\mathcal{R}} )$ is an atlas in $\mathbb{R}^N$, and $M\in [0,\infty[$. Let   $ f =\sum_{k=1}^sf_k  \in  \mathcal{\widetilde{L}}^{\lambda ,\mathcal R}_{p, \gamma }(\Omega)$.  Then, there exists $F=\sum_{k=1}^sF_k$ belonging to    $ \mathcal{{L}}^{\lambda ,\mathcal R}_{p, \gamma }(\mathbb{R}^N)$ such that $F_{|\Omega}=f$. Moreover, 
\begin{equation}\label{ext_elementary_hold0}    
\sum_{k=1}^{s}
   \| F_k \|_{ {\mathcal L}^{\lambda, R_k }_{p, \gamma }(\R^N) }  \le C \sum_{k=1}^{s}
   \| f_k \|_{ {\mathcal L}^{\lambda, R_k }_{p, \gamma }(\Omega  ) }  
 \end{equation}
where $C$ is a positive constant  depending  only on $N, \mathcal{A},M, p,\lambda, \gamma$.
\end{lemma}
\begin{proof}
First of all, we note that if $\Omega_k$  is a patch of $\Omega$ with $\Omega_k\cap \partial \Omega \ne\emptyset$ as in Definition~\ref{atlas}, then $\Omega_k$   can be naturally included in a domain $\widetilde{\Omega}_k$  which is isometric to an elementary  domain of class $C^{0,\gamma}$ as those discussed in the previous section. Indeed, the function $\varphi_k$ describing the boundary of $R_k(\Omega_k)$ can be extended to the whole of ${\mathbb{R}}^{N-1}$  by the  formula
\begin{equation}
\widetilde{\varphi}_k(\bar x ) = \inf_{\bar{y}\in W}\left\{\varphi_k(\bar{y})+  {\rm Lip}_{\gamma}\varphi_k \,  |\bar{x}-\bar{y}|^{\gamma}\right\},
    \end{equation} 
preserving the $C^{0,\gamma}$ seminorm  ${\rm Lip}_{\gamma}\varphi_k$ (see  e.g. \cite[Thm.~1.8.3]{bib4}). Then, we can define the domain $\widetilde{\Omega}_k$  by setting 
$$
R_k(\widetilde{\Omega}_k)=\{(\bar x, x_N):\ x_N<\widetilde{\varphi}_k(\bar x)    \}\, .
$$

Assume now that 
$f=\sum_{k=1}^{s}f_k$, $f_k   \in{\mathcal{ \widetilde L}^{\lambda,k}_{p,\gamma }}(\Omega_{k})$ as  in \eqref{patchyspace}.  Fix now $k=1, \dots , s$  such that   the patch $\Omega_k$ satisfies
$\Omega_k\cap \partial \Omega \ne\emptyset$. Since  $
 {\rm supp} f_k \subset  \overline{( V_k )_{\delta /2}}
 $, by proceeding as in the proof of Lemma~\ref{compactlem} we see
 that the function  $f_{k,0}$ obtained by extending by zero the function $f_k$ to the whole of $\widetilde{\Omega}_k$ belongs to 
 $\mathcal{L}^{\lambda}_{p,\gamma}(\widetilde{\Omega}_k)$ and  
\begin{equation}\label{ext000}
|f_{k,0}|_{\mathcal{L}^{\lambda,R_k}_{p,\gamma}(\widetilde{\Omega}_k)}\le C\|f_k  \|_{{\mathcal{L}}^{\lambda, R_k}_{p,\gamma}(\Omega_{k})}
\end{equation} 
for some positive constant $C$ independent of $f_k$  (note that the distance of the support of each function $f_k$ from the boundary of the cuboid $V_k$ is estimated from below by $\delta/2$). 
Now, 
since $R_k(\widetilde{\Omega}_k)$ is a $C^{0,\gamma}$ elementary domain as in the previous section, Theorem~\ref{estensioneholder} provides a function $\widetilde{F}_k$ defined in the whole of ${\mathbb{R}}^N $ such that $\widetilde{F}_k \circ R_k^{-1}$ belongs to $  \mathcal{L}^{\lambda}_{p,\gamma}(\R^N)$ and extends  $f_{k,0}\circ R_k^{-1}$.
Now, let's consider functions $\psi_k\in C^{\infty}_c(\R^N)$ such that  
\begin{equation}\label{psi}
   \psi_k:=\begin{cases} 1, & {\rm in}\ (R_k(V_k))_{\delta /2} \\ 0, &  {\rm in}\ S_{k},
\end{cases}
\end{equation} 
where  
$S_{W}:=\begin{matrix}\prod_{i=1}^{N-1}]a_i,b_i[\times ]b_N-\delta /2,+\infty[ \end{matrix}$.

Then, we define a function $G_k$ by setting  $G_k=(\widetilde{F}_k \circ R_k^{-1}) \psi_k$.  

 If $\Omega_k\subset \Omega$ then we set $G_k= f_{k,0}\circ   R_k^{-1}$.  

Thus by Lemma~\ref{product} the  function 
\begin{equation}\label{T_f}
F:=\sum_{k=1}^{s} G_k\circ R_k
\end{equation}
belongs to $ {\mathcal L}^{\lambda, {\mathcal{R}} }_{p, \gamma } (  \R^N)$. 
 
 We claim  that  $F_{|\Omega}=f.$  Indeed, we note that ${\rm supp} (G_{k} \circ R_k) \subset V_k$. In particular, if $x\in \Omega$ and 
 $G_{k} ( R_k (x) )\ne 0$ then $x\in \Omega_k$ hence $G_{k} ( R_k (x) )=f_k(x) $; on the other hand if  $G_{k} ( R_k (x) )=0 $ then also $f_k(x)=0$.  
 Thus we conclude that 
 \begin{equation}
 \sum_{k=1}^s G_{k} ( R_k (x) )= \sum_{k=1}^s f_k(x) = f(x)
 \end{equation}
 and the claim is proved.  
 
   Finally,  by \eqref{ext000} and Lemmas~\ref{product} and \ref{restrizioneholder}  we get
\begin{equation}\label{ineq_1}\sum_{k=1}^{s}\| G_k\circ R_k  \|_{\mathcal{L}^{\lambda,R_k}_{p,\gamma}(\R^N)}\le C\sum_{k=1}^{s}\|f_{k} \|_{\mathcal{L}^{\lambda,R_k}_{p,\gamma}(\Omega_k)}\le C\sum_{k=1}^{s}\|f_{k} \|_{\mathcal{L}^{\lambda,R_k}_{p,\gamma}(\Omega )}
\end{equation}
and the proof is complete. 
\end{proof}

Let $\Omega$ be a bounded domain of class $C^{0,\gamma}_{M}(\mathcal{A})$.  By a partition of unity associated with $\mathcal{V}$ we understand a family of $C^{\infty}$ functions  $\Psi:=\left\{\psi_k\right\}_{k=1}^{s}$ such that 
$0\le \psi_k\le 1$, supp$(\psi_k)\subset (V_k)_{\delta /2}$ for all $k=1\dots, s$, and  $\sum_{k=1}^s \psi_k=1$ in $\overline{\bigcup^{s}_{k=1}(V_k)_{\delta }}$, hence 
  $\sum_{k=1}^s \psi_k=1$ in $\Omega$. 
The existence of $\Psi$ is well-known. 
Note that  for any real-valued function $f$ defined on $\Omega$ we have 
\begin{equation}\label{decstandard}
f=\sum_{k=1}^s f \psi _k .
\end{equation}
Then we define the space 
 $\widetilde{\mathcal{L}}^{\lambda,\mathcal{R}}_{p,\gamma, \Psi}(\Omega)$  by setting
 \begin{equation}
 \widetilde{\mathcal{L}}^{\lambda,\mathcal{R}}_{p,\gamma, \Psi}(\Omega):=\left\{f\in L^{p}(\Omega):f\psi_k\in \widetilde{{\mathcal L}}^{\lambda, R_k }_{p,\gamma }(\Omega)\ \forall k= 1,\dots , s\right\}
\end{equation}
and we endow it with the norm defined by
$$
\|f \|_{   \widetilde{\mathcal{L}}^{\lambda,\mathcal{R}}_{p,\gamma, \Psi}(\Omega) }=\sum_{k=1}^s   \| f \psi_k \|_{ {\mathcal L}^{\lambda, R_k }_{p, \gamma }(\Omega) } \, .
 $$
Note that $ \widetilde{\mathcal{L}}^{\lambda,\mathcal{R}}_{p,\gamma, \Psi}(\Omega) \subset  \widetilde{\mathcal{L}}^{\lambda,\mathcal{R}}_{p,\gamma}(\Omega)$.

By  Lemma~\ref{ext_elementary_hold} we can easily deduce the following corollary.

\begin{teorema}
Let $\Omega$ be a bounded domain  of class $C^{0,\gamma}_{M}(\mathcal{A})$ where  ${\mathcal{A} }=(s,\delta , {\mathcal{V}}, {\mathcal{R}} )$ is an atlas in $\mathbb{R}^N$, $M\in [0,\infty [$. Let  $\Psi:=\left\{\psi_k\right\}_{k=1}^{s}$ be a partition of unity associated with  the family of cuboids $\mathcal{V}=\left\{V_k\right\}^{s}_{k=1}$ as above. Then there exists a linear continuous operator $T$ from 
$\widetilde{\mathcal{L}}^{\lambda,\mathcal{R}}_{p,\gamma, \Psi}(\Omega )$ to   $\mathcal{L}^{\lambda,\mathcal{R}}_{p,\gamma}(\mathbb{R}^N)$ such that 
$(Tf)(x) =f(x)$ for all $x\in \Omega$.
\end{teorema}

\begin{proof} Given $f\in  \widetilde{\mathcal{L}}^{\lambda,\mathcal{R}}_{p,\gamma, \Psi}(\Omega ) $, we can write it as in \eqref{decstandard}, hence
$f=\sum_{k=1}^sf_k$ where $f_k=f\psi_k\in \widetilde{{\mathcal L}}^{\lambda, R_k }_{p,\gamma }(\Omega)$. Then the function $Tf=F$ defined as in 
\eqref{T_f} is an extension of $f$. Note that by Lemma~\ref{ext_elementary_hold}, $Tf$ can be written in the form $Tf=\sum_{k=1}^sF_k$ and estimate
\eqref{ext_elementary_hold0}  holds. By passing to the infimum in the right-hand side of \eqref{ext_elementary_hold0} with respect to all admissible decompositions of $Tf$ we conclude that 
$$
\| Tf \|_{ {\mathcal{L}}^{\lambda,\mathcal{R}}_{p,\gamma }(\R^N )}  \le C \| f\| _{   \widetilde{\mathcal{L}}^{\lambda,\mathcal{R}}_{p,\gamma, \Psi}(\Omega) }. 
$$
Since the operator $Tf$ is linear by definition, the previous inequality yields its continuity. 
\end{proof}

\begin{remark}\label{patcheq} In the Lipschitz case $\gamma =1$  the space $\widetilde{\mathcal{L}}^{\lambda,\mathcal{R}}_{p,\gamma, \Psi} (\Omega)$  is just the classical Campanato space  ${\mathcal{L}}^{\lambda}_{p}(\Omega)$ defined in \eqref{clasintro} which coincides with  ${\mathcal{L}}^{\lambda}_{p,1}(\Omega)$. Indeed, by Lemma~\ref{productbis} it follows that if $f\in {\mathcal{L}}^{\lambda}_{p,1}(\Omega) $ then $f\psi_k \in {\mathcal{L}}^{\lambda}_{p,1}(\Omega)  $ and since the support of $f\psi_k$ is contained in $\overline{ (V_j)_{\delta/2}}$ we have that $f\psi_k\in  { \widetilde{ {\mathcal  L} }}  ^{\lambda, R_k }_{p, 1 }(\Omega) $. This implies that 
$f\in \widetilde{\mathcal{L}}^{\lambda,\mathcal{R}}_{p,1, \Psi} (\Omega)$. Accordingly ${\mathcal{L}}^{\lambda}_{p,1}(\Omega)\subset 
\widetilde{\mathcal{L}}^{\lambda,\mathcal{R}}_{p,1, \Psi} (\Omega)$. The other inclusion is obvious. 
The  same arguments show also  that 
$
  \mathcal{L}^{\lambda ,\mathcal R}_{p, 1 }(\Omega)  = {\mathcal{L}}^{\lambda}_{p,1}(\Omega) 
  $
  and
  $
  \mathcal{L}^{\lambda,\mathcal{R}}_{p,1}(\mathbb{R}^N)  = \mathcal{L}^{\lambda }_{p,1}(\mathbb{R}^N)$.
\end{remark}

\vspace{12pt}

We conclude this paper  by formulating the following problem. \\

{\bf Problem:} Given a domain $\Omega$ in $\R^N$ and  $\gamma \in ]0,1]$,   let
\begin{equation} \label{bmogamma}
\bmo_{\gamma} (\Omega ) =\left\{ f\in L^1_{loc}(\Omega):\     |f|_{\bmo_{\gamma} (\Omega )} \ne \infty   \right\}
\end{equation}
where 
$$ \sup_{B_{\gamma }\subset \Omega}   \dashint_{B_{\gamma}} \left|f(y)    
-\dashint_{B_{\gamma}}   f(z)dz  \right|dy
$$ 
and the supremum is taken on all balls $B_{\gamma}$ with respect to the metric $\delta_{\gamma}$  contained in $\Omega$.  The limiting case corresponding to $\gamma =1$ gives the classical  space $\bmo (\Omega)$ defined in Section~\ref{bmosec}.  Note that in general the space 
$  \bmo_{\gamma} (\Omega ) $ is larger than the Campanato space $\mathcal {L}^{1 }_{1,\gamma} (\Omega)  $, see Example~\ref{strip}.  

For any $\gamma \in ]0,1[ $, characterise the open sets  $\Omega$ such that the   functions of the space $\bmo_{\gamma} (\Omega )$ can be extended to the whole of $\R^N$ by functions belonging to $\bmo_{\gamma} (\R^N )$.  This problem was solved for $\gamma =1$ in \cite{jones}. \\

{\bf Acknowledgments}  The authors are  thankful to Professor Andrea Cianchi for bringing to their attention item \cite{gallia} in the reference list and to  Professor Massimo Lanza de Cristoforis for pointing out \cite[Lemma~11]{bolasi}.  The authors are also thankful to Professors Victor Burenkov and Vincenzo Vespri for useful discussions.  P.D. Lamberti  is a  member of the Gruppo Nazionale per l'Analisi Matematica, la Probabilit\`a e le loro Applicazioni (GNAMPA) of the Istituto Nazionale di Alta Matematica (INdAM).

\clearpage


\begin{thebibliography}{} 
\bibitem{aalto} D. Aalto, L. Berkovits, O. E. Kansanen; H. Yue., John Nirenberg Lemmas for a doubling measure. Studia Mathematica, 204(1), 21-37 (2011).
\bibitem{barozzi} G.C.  Barozzi,  Su una generalizzazione degli spazi $L^{q,\lambda}$ di Morrey. 
Ann. Scuola Norm. Sup. Pisa Cl. Sci. (3) 19 (1965), 609-626. 
\bibitem{beveza} C. Bernardini, V. Vespri, M. Zaccaron, A note on Campanato's 
$L^p$-regularity with continuous coefficients. Eurasian Mathematical Journal, Volume 13, Number 2, (2022), 44-53. 
\bibitem{bolk} M. Bolkart, Martin, Y. Giga, T. Suzuki, Y.  Tsutsui, 
Equivalence of BMO-type norms with applications to the heat and Stokes semigroups.
Potential Anal. 49 (2018), no. 1, 105-130. 
\bibitem{bolasi}  G. Bourdaud, M. Lanza de Cristoforis, W. Sickel, Functional calculus on BMO and related spaces. J. Funct. Anal. 189 (2002), no. 2, 515-538. 
\bibitem{bib10} H. Brezis; L. Nirenberg, Degree Theory and BMO; Part II: Compact Manifolds with Boundaries. Selecta Mathematica, New Series, Vol. 2, No. 3 (1996), 309-368.
\bibitem{bu}  V.I. Burenkov, V. I. The continuation of functions with preservation and with deterioration of their differential properties. (Russian) Dokl. Akad. Nauk SSSR 224 (1975), no. 2, 269-272. 
\bibitem{gallia} A. Butaev, G. Dafni, Approximation and extension of functions of vanishing mean oscillation. 
J. Geom. Anal. 31 (2021), no. 7, 6892-6921. 
\bibitem{campa63} S. Campanato, Proprietà di h\"olderianità di alcune classi di funzioni. (Italian) Ann. Scuola Norm. Sup. Pisa ($3$) 17 ($1963$), 175-188.
\bibitem{campa64} S. Campanato, Proprietà di una famiglia di spazi funzionali. (Italian) Ann. Scuola Norm. Sup. Pisa ($3$) 18 ($1964$), 137-160. 
\bibitem{fanlam} M.S. Fanciullo; P.D. Lamberti,  On Burenkov's extension operator preserving Sobolev-Morrey spaces on Lipschitz domains.
Math. Nachr. 290 (2017), no. 1, 37-49. 
\bibitem{dapra} G. Da Prato, Spazi $\mathcal{L}^{p,\theta}(\Omega,\delta)$ e le loro propriet\`{a}. (Italian) Ann. Mat. Pura Appl. (4) 69 (1965), 383-392.
\bibitem{bib2} Y. Gotoh, On decomposition theorem for BMO and VMO. Complex Variables and Elliptic Equations, 43:1 (2000), 59-76, DOI: 10.1080/17476930008815301.
\bibitem{bib6} F. John; L. Nirenberg,  On function of Bounded Mean Oscillation. Comm.
Pure Appl. Math. Vol. 14 (1961), 415-426.
 \bibitem{jones} P.W. Jones,  Extension theorems for BMO. Indiana Univ. Math. J. 29 (1980), no. 1, 41-66.
\bibitem{bib4} A. Kufner, O. John, S. Fucik, Function Spaces. Monographs and Textbooks on Mechanics of Solids and Fluids; Mechanics: Analysis. Noorhoff International Publishing, Leyden; Academia, Prague, 1977.
\bibitem{lamves} P.D. Lamberti, V. Vespri, Remarks on Sobolev-Morrey-Campanato spaces defined on $C^{0,\gamma}$ domains, 
Eurasian Math. J.,  Volume 10,  (2019), Number 4, 47-62.
\bibitem{lamvio} P.D. Lamberti, I.Y. Violo, On Stein's extension operator preserving Sobolev-Morrey spaces. 
Math. Nachr. 292 (2019), no. 8, 1701-1715. 
\bibitem{samko} H. Rafeiro, N. Samko, S.  Samko, Stefan, Morrey-Campanato spaces: an overview. Operator theory, pseudo-differential equations, and mathematical physics, 293-323, Oper. Theory Adv. Appl., 228, Birkh\"{a}user/Springer Basel AG, Basel, 2013. 
\bibitem{sawano} Y. Sawano, G. Di Fazio, D.I. Hakim,  Morrey spaces--introduction and applications to integral operators and PDE's. Vol. I.II. Monographs and Research Notes in Mathematics. CRC Press, Boca Raton, FL, 2020.
\bibitem{bib12} E.M. Stein, Harmonic Analysis: Real-variable Methods, Orthogonality and Oscillatory Integrals. Princeton, NJ: Princeton University Press, 1993.
\bibitem{stri}R.S. Strichartz, The Hardy space $H^1$ on manifolds and submanifolds. Can. J. Math., Vol. XXIV, (1972),  No. 5, pp. 915-925.
\bibitem{bib13} R. Toledano, A note on the Lebesgue
differentiation theorem in spaces
of homogeneous type. Real Analysis Exchange
Vol. 29(1), (2003/2004), pp. 335-340.

    
    
   
\end{thebibliography}
\end{document}